\documentclass{article}

\usepackage{graphicx}
\usepackage{multirow}
\usepackage{amsmath,amssymb,amsfonts}
\usepackage{amsthm}
\usepackage[title]{appendix}
\usepackage{xcolor}
\usepackage{textcomp}
\usepackage{manyfoot}
\usepackage{bookmark}
\usepackage{booktabs}
\usepackage{listings}
\usepackage{indentfirst}
\usepackage{epstopdf}
\usepackage{mathtools}
\usepackage{url}
\usepackage{xspace}
\usepackage{rotating}
\usepackage{scalefnt}
\usepackage{upgreek}
\usepackage{enumitem}
\usepackage{array}
\usepackage{geometry}
\usepackage{placeins}
\usepackage{caption}

\usepackage[T1]{fontenc}
\usepackage[utf8]{inputenc}

\definecolor{keywordcolor}{rgb}{0.7, 0.1, 0.1}
\definecolor{tacticcolor}{rgb}{0.0, 0.1, 0.3}
\definecolor{commentcolor}{rgb}{0.4, 0.4, 0.4}
\definecolor{stringcolor}{rgb}{0.5, 0.3, 0.2}
\definecolor{symbolcolor}{rgb}{0.1, 0.2, 0.7}
\definecolor{sortcolor}{rgb}{0.1, 0.5, 0.1}
\definecolor{attributecolor}{rgb}{0.7, 0.1, 0.1}
\definecolor{errorcolor}{rgb}{1, 0, 0}

\lstdefinelanguage{lean}{
mathescape=false,
texcl=false,
morekeywords=[1]{
import, prelude, protected, private, noncomputable, definition, meta, renaming,
hiding, exposing, parameter, parameters, begin, conjecture, constant, constants,
hypothesis, lemma, corollary, variable, variables, premise, premises, theory,
print, theorem, proposition, example, open, as, export, override, axiom, axioms,
inductive, with, without, structure, record, universe, universes, alias, help,
precedence, reserve, declare_trace, add_key_equivalence, match, infix, infixl,
infixr, notation, postfix, prefix, instance, vm_eval, check, coercion, end,
this, suppose, using, using_well_founded, namespace, section, fields, attribute,
local, set_option, extends, include, omit, classes, class, instances, coercions,
attributes, raw, replacing, calc, have, show, suffices, by, in, at, let, forall,
Pi, fun, exists, if, dif, then, else, assume, take, obtain, from, aliases,
register_simp_ext, mutual, do, def, run_command, to, abbrev, where
},
morekeywords=[2]{Sort, Type, Prop, Type*, Type₀, Type₁, Type₂, Type₃},
morekeywords=[3]{sorry, admit},
literate=
{α}{{\ensuremath{\mathrm{\upalpha}}}}1
{β}{{\ensuremath{\mathrm{\upbeta}}}}1
{γ}{{\ensuremath{\mathrm{\upgamma}}}}1
{δ}{{\ensuremath{\mathrm{\updelta}}}}1
{ε}{{\ensuremath{\mathrm{\varepsilon}}}}1
{ζ}{{\ensuremath{\mathrm{\zeta}}}}1
{η}{{\ensuremath{\mathrm{\eta}}}}1
{θ}{{\ensuremath{\mathrm{\theta}}}}1
{ι}{{\ensuremath{\mathrm{\iota}}}}1
{κ}{{\ensuremath{\mathrm{\kappa}}}}1
{λ}{{\color{symbolcolor}\ensuremath{\mathrm{\uplambda}}}}1
{μ}{{\ensuremath{\mathrm{\mu}}}}1
{ν}{{\ensuremath{\mathrm{\nu}}}}1
{ξ}{{\ensuremath{\mathrm{\xi}}}}1
{π}{{\ensuremath{\mathrm{\mathnormal{\pi}}}}}1
{ρ}{{\ensuremath{\mathrm{\rho}}}}1
{σ}{{\ensuremath{\mathrm{\sigma}}}}1
{τ}{{\ensuremath{\mathrm{\tau}}}}1
{χ}{{\ensuremath{\mathrm{\chi}}}}1
{ψ}{{\ensuremath{\mathrm{\psi}}}}1
{ω}{{\ensuremath{\mathrm{\omega}}}}1
{υ}{{\ensuremath{\mathrm{\nu}}}}1
{Γ}{{\ensuremath{\mathrm{\Gamma}}}}1
{Δ}{{\ensuremath{\mathrm{\Delta}}}}1
{Θ}{{\ensuremath{\mathrm{\Theta}}}}1
{Λ}{{\ensuremath{\mathrm{\Lambda}}}}1
{Σ'}{{\color{symbolcolor}\ensuremath{\Sigma'}}}2
{Σ}{{\color{symbolcolor}\ensuremath{\Sigma}}}1
{Ξ}{{\ensuremath{\mathrm{\Xi}}}}1
{Ψ}{{\ensuremath{\mathrm{\Psi}}}}1
{Ω}{{\ensuremath{\mathrm{\Omega}}}}1
{Φ}{{\ensuremath{\mathrm{\Phi}}}}1
{Π}{{\color{symbolcolor}\ensuremath{\mathrm{\Uppi}}}}1
{ℵ}{{\ensuremath{\aleph}}}1
{ℕ}{{\ensuremath{\mathbb{N}}}}1
{ℤ}{{\ensuremath{\mathbb{Z}}}}1
{ℝ}{{\ensuremath{\mathbb{R}}}}1
{ℚ}{{\ensuremath{\mathbb{Q}}}}1
{ℂ}{{\ensuremath{\mathbb{C}}}}1
{ℓ}{{\ensuremath{\ell}}}1
{⊤}{{\ensuremath{\top}}}1
{⊥}{{\color{symbolcolor}\ensuremath{\perp}}}1
{∞}{{\color{symbolcolor}\ensuremath{\infty}}}1
{√}{{\color{symbolcolor}\ensuremath{\sqrt}}}1
{≤}{{\color{symbolcolor}\ensuremath{\leq}}}1
{≥}{{\color{symbolcolor}\ensuremath{\geq}}}1
{≠}{{\color{symbolcolor}\ensuremath{\neq}}}1
{≈}{{\color{symbolcolor}\ensuremath{\approx}}}1
{≡}{{\color{symbolcolor}\ensuremath{\equiv}}}1
{≃o}{{\color{symbolcolor}\ensuremath{\simeq_o}}}2
{≃}{{\color{symbolcolor}\ensuremath{\simeq}}}1
{∂}{{\color{symbolcolor}\ensuremath{\partial}}}1
{∆}{{\color{symbolcolor}\ensuremath{\triangle}}}1
{∫}{{\color{symbolcolor}\ensuremath{\int}}}1
{∑}{{\color{symbolcolor}\ensuremath{\Sigma}}}1
{∓}{{\color{symbolcolor}\ensuremath{\mp}}}1
{±}{{\color{symbolcolor}\ensuremath{\pm}}}1
{×}{{\color{symbolcolor}\ensuremath{\times}}}1
{⊕ₗ}{{\color{symbolcolor}\ensuremath{\oplus_{\ell}}}}2
{⊕}{{\color{symbolcolor}\ensuremath{\oplus}}}1
{⊗}{{\color{symbolcolor}\ensuremath{\otimes}}}1
{⊞}{{\color{symbolcolor}\ensuremath{\boxplus}}}1
{⨆}{{\ensuremath{\bigsqcup}}}1
{⨅}{{\color{symbolcolor}\ensuremath{\bigsqcap}}}1
{∇}{{\color{symbolcolor}\ensuremath{\nabla}}}1
{⬝}{{\color{symbolcolor}\ensuremath{\cdot}}}1
{•}{{\color{symbolcolor}\ensuremath{\cdot}}}1
{·}{{\color{symbolcolor}\ensuremath{\cdot}}}1
{∘}{{\color{symbolcolor}\ensuremath{\circ}}}1
{∧}{{\color{symbolcolor}\ensuremath{\wedge}}}1
{∨}{{\color{symbolcolor}\ensuremath{\vee}}}1
{¬}{{\color{symbolcolor}\ensuremath{\neg}}}1
{⊢}{{\color{symbolcolor}\ensuremath{\vdash}}}1
{∀}{{\color{symbolcolor}\ensuremath{\forall}}}1
{∃}{{\color{symbolcolor}\ensuremath{\exists}}}1
{↦}{{\color{symbolcolor}\ensuremath{\mapsto}}}1
{→}{{\color{symbolcolor}\ensuremath{\rightarrow}}}1
{←}{{\color{symbolcolor}\ensuremath{\leftarrow}}}1
{↔}{{\color{symbolcolor}\ensuremath{\leftrightarrow}}}1
{⇒}{{\color{symbolcolor}\ensuremath{\Rightarrow}}}1
{⟹}{{\color{symbolcolor}\ensuremath{\Longrightarrow}}}1
{⇐}{{\color{symbolcolor}\ensuremath{\Leftarrow}}}1
{⟸}{{\color{symbolcolor}\ensuremath{\Longleftarrow}}}1
{⟶}{{\color{symbolcolor}\ensuremath{\longrightarrow}}}1
{⥤}{{\color{symbolcolor}\ensuremath{\Rightarrow}}}1
{∩}{{\color{symbolcolor}\ensuremath{\cap}}}1
{∪}{{\color{symbolcolor}\ensuremath{\cup}}}1
{⊂}{{\color{symbolcolor}\ensuremath{\subseteq}}}1
{⊆}{{\color{symbolcolor}\ensuremath{\subseteq}}}1
{⊄}{{\color{symbolcolor}\ensuremath{\nsubseteq}}}1
{⊈}{{\color{symbolcolor}\ensuremath{\nsubseteq}}}1
{⊃}{{\color{symbolcolor}\ensuremath{\supseteq}}}1
{⊇}{{\color{symbolcolor}\ensuremath{\supseteq}}}1
{⊅}{{\color{symbolcolor}\ensuremath{\nsupseteq}}}1
{⊉}{{\color{symbolcolor}\ensuremath{\nsupseteq}}}1
{∈}{{\color{symbolcolor}\ensuremath{\in}}}1
{∉}{{\color{symbolcolor}\ensuremath{\notin}}}1
{∋}{{\color{symbolcolor}\ensuremath{\ni}}}1
{∌}{{\color{symbolcolor}\ensuremath{\notni}}}1
{∅}{{\color{symbolcolor}\ensuremath{\emptyset}}}1
{∖}{{\color{symbolcolor}\ensuremath{\setminus}}}1
{†}{{\color{symbolcolor}\ensuremath{\dag}}}1
{⋆}{{\ensuremath{\star}}}1
{∥}{{\ensuremath{\|}}}1
{φ}{{\ensuremath{\phi}}}1
{⌞}{{\ensuremath{\llcorner}}}1
{⌟}{{\ensuremath{\lrcorner}}}1
{⦃}{{\ensuremath{\{\!|}}}1
{⦄}{{\ensuremath{|\!\}}}}1
{⟨}{{\ensuremath{\langle}}}1
{⟩}{{\ensuremath{\rangle}}}1
{⟪}{{\ensuremath{\llangle}}}1
{⟫}{{\ensuremath{\rrangle}}}1
{⌊}{{\ensuremath{\lfloor}}}1
{⌋}{{\ensuremath{\rfloor}}}1
{⌈}{{\ensuremath{\lceil}}}1
{⌉}{{\ensuremath{\rceil}}}1
{‹}{{\guilsinglleft}}1
{›}{{\guilsinglright}}1
{｜}{{\ensuremath{\mid}}}1
{₀₀}{{\ensuremath{_{00}}}}2
{₁₁}{{\ensuremath{_{11}}}}2
{₁₂}{{\ensuremath{_{12}}}}2
{₂₁}{{\ensuremath{_{21}}}}2
{₂₂}{{\ensuremath{_{22}}}}2
{₀}{{\ensuremath{_0}}}1
{₁}{{\ensuremath{_1}}}1
{₂}{{\ensuremath{_2}}}1
{₃}{{\ensuremath{_3}}}1
{₄}{{\ensuremath{_4}}}1
{₅}{{\ensuremath{_5}}}1
{₆}{{\ensuremath{_6}}}1
{₇}{{\ensuremath{_7}}}1
{₈}{{\ensuremath{_8}}}1
{₉}{{\ensuremath{_9}}}1
{₊}{{\ensuremath{_+}}}1
{ₐ}{{\ensuremath{_a}}}1
{ᵢ}{{\ensuremath{_i}}}1
{ⱼ}{{\ensuremath{_j}}}1
{ₗ}{{\ensuremath{_l}}}1
{ₙ}{{\ensuremath{_n}}}1
{ₘ}{{\ensuremath{_m}}}1
{ₛ}{{\ensuremath{_s}}}1
{ᵥ}{{\ensuremath{_v}}}1
{ᵀ}{{\ensuremath{^\top}}}1
{ᴴ}{{\ensuremath{^H}}}1
{∣}{{\color{symbolcolor}\ensuremath{\mid}}}1
{ᵒᵖ}{{\color{symbolcolor}\textsuperscript{op}}}2
{ᗮ}{{\ensuremath{^\perp}}}1
{¹}{{\ensuremath{^1}}}1
{ᶠ}{{\ensuremath{^f}}}1
{ˡ}{{\ensuremath{^l}}}1
{ⁱ}{{\ensuremath{^i}}}1
{⁻}{{\ensuremath{^{-}}}}1
{↑}{{\color{symbolcolor}\ensuremath{\uparrow}}}1
{↓}{{\color{symbolcolor}\ensuremath{\downarrow}}}1
{//}{{\color{symbolcolor}//}}2
{<|>}{{\color{symbolcolor}<|>}}3
{...}{{\ensuremath{\ldots}}}3
{\$}{{\color{symbolcolor}\$}}1
{:}{{\color{symbolcolor}:}}1
{|}{{\color{symbolcolor}|}}1
{=}{{\color{symbolcolor}=}}1
{<}{{\color{symbolcolor}<}}1
{>}{{\color{symbolcolor}>}}1
{+}{{\color{symbolcolor}+}}1
{*}{{\color{symbolcolor}*}}1
{`}{{\ensuremath{{}^\backprime}}}1
{'}{{\ensuremath{{}^\prime}}}1,
morecomment=[s]{/-}{-/},
morecomment=[l]{--},
showstringspaces=false,
keepspaces=true,
morestring=[b]{"},
tabsize=3,
extendedchars=true,
inputencoding=utf8,
sensitive=true,
breaklines=true,
breakatwhitespace=true,
lineskip={-1.5pt},
basicstyle={\ttfamily\small},
captionpos=b,
columns=[l]fullflexible,
identifierstyle={\ttfamily\color{black}},
keywordstyle=[1]{\ttfamily\color{keywordcolor}},
keywordstyle=[2]{\ttfamily\color{sortcolor}},
keywordstyle=[3]{\ttfamily\color{errorcolor}},
stringstyle={\ttfamily\color{stringcolor}},
commentstyle={\ttfamily\itshape\color{commentcolor}},
}

\newcommand{\lean}[1]{\lstinline[language=lean, mathescape=true]{#1}}
\newcommand{\repourl}{https://github.com/wl-ma/MatDecompFormal/blob/13540d9fdefff347f8698a2e32e3e4bfdbb384e9}
\newcommand{\mathliburl}{https://github.com/leanprover-community/mathlib4/blob/0a2070811b139d537c0d33a1bbf2a946bb7e60c4}
\newcommand{\mathlib}{\textsf{mathlib}\xspace}
\newcommand{\breaktt}[1]{\begingroup\ttfamily\nolinkurl{#1}\endgroup}

\newenvironment{tablenotes}{\list{}{\setlength{\labelsep}{0.5em}%
\setlength{\labelwidth}{1em}%
\setlength{\leftmargin}{1.5em}%
\setlength{\rightmargin}{0pt}%
\setlength{\topsep}{-6pt}%
\setlength{\itemsep}{2pt}%
\setlength{\partopsep}{0pt}%
\setlength{\listparindent}{0em}%
\setlength{\parsep}{0pt}}%
\item\relax%
}{\endlist}

\newcommand{\tnote}[1]{$^{#1}$}

\theoremstyle{plain}
\newtheorem{theorem}{Theorem}

\theoremstyle{remark}

\theoremstyle{definition}

\begin{document}

\title{A Unified Framework for Formalizing Matrix Decomposition Proofs}

\author{
\begin{tabular}{c}
Wanli Ma\textsuperscript{1,*} \quad
Zichen Wang\textsuperscript{2,*,\textdagger} \quad
Zaiwen Wen\textsuperscript{1}
\\[0.6em]
\small\textsuperscript{1}{\centering Beijing International Center for Mathematical Research, Peking University}
\\[0.35em]
\small\textsuperscript{2}{\centering School of Mathematical Sciences, Peking University}
\\[0.6em]
\footnotesize Emails: \texttt{wlma@pku.edu.cn}; \texttt{zichenwang25@stu.pku.edu.cn}; \texttt{wenzw@pku.edu.cn}.
\end{tabular}
}
\date{}
\maketitle
\begingroup
\renewcommand{\thefootnote}{\fnsymbol{footnote}}
\footnotetext[1]{These authors contributed equally to this work.}
\footnotetext[2]{Corresponding author.}
\endgroup

\begin{abstract}
Existence proofs for many matrix decompositions share a recursive routine: a local transformation prepares the matrix, a slice is selected, a recursive solution is obtained, and the result is lifted and transported back. Formalizing this routine uniformly in dependent type theory is difficult because recursive subproblems may change index types, and reconstruction must preserve structural predicates across block embeddings and reindexings. We develop a Lean~4 framework that separates decomposition schemas, transformations, reduction strategies, measures, lifting, transport, and subtype induction. The framework uses general index types, packages square and rectangular matrices in universe types, and provides a decomposition driver that assembles strategy data into subtype-induction instances. It has been instantiated across PLU, LU, LDL/Cholesky, QR variants, Gauss rank normal form, Hessenberg reductions, Schur variants, normal spectral decomposition, SVD, bidiagonalization, tridiagonalization, UTV, Smith normal form, rational canonical form, and Jordan-type forms at varying levels of statement strength. Across these instances, repeated decomposition proofs are best treated not as separate tasks but as instances of a more general inductive statement whose interface records a certified proof path compatible with the chosen decomposition statement.
\end{abstract}


\section{Introduction}\label{sec:intro}
Matrix decompositions are fundamental in linear algebra and scientific computing. They underlie the solution of linear systems, least-squares problems, eigenvalue computations, and low-rank approximation, and they provide the structural basis for many classical numerical algorithms \cite{stewart1998matrix,trefethen2022numerical,golub2013matrix,higham2002accuracy,demmel1997applied,watkins2002fundamentals}. Lean~4 and its mathematical library \mathlib provide substantial foundations for algebra, analysis, topology, and linear algebra \cite{mathlib2020,avigad2020mathematics,moura2021lean}, making matrix decompositions a natural target for formalization. Despite their different concrete forms, existence proofs for these decompositions often share a compact recursive shape: apply a local normalization, reduce to a smaller subproblem, solve it recursively, and lift the result back through block structure \cite{stewart1998matrix,trefethen2022numerical,golub2013matrix}. Similar proof structures appear in matrix reduction procedures, constructions of normal forms, and related results in linear algebra \cite{stewart1998matrix,golub2013matrix,horn2012matrix,gantmacher1959matrix,lancaster1985theory,aransay2016formalisation,divason2022formalization}. The recurrence of this pattern suggests that the common proof organization itself can be represented as formal structure.

Related formalization projects have already shown that substantial mathematical results can be developed rigorously in proof assistants \cite{geuvers2009proof,nipkow2002isabelle,bertot2004interactive,mahboubi2021mathematical,wiedijk2006seventeen,harrison1996formalized,avigad2024formal,buzzard2022point}. In linear algebra, existing developments include formalizations of matrix echelon forms, Smith normal form in both Isabelle/HOL and Coq, and Jordan normal form with spectral radius theory \cite{aransay2016formalisation,divason2022formalization,cano2016formalized,thiemann2016jordan}. Further work covers QR decomposition \cite{divason2015qr} and the fundamental theorem of linear algebra \cite{divason2017fundamental}. Large-scale projects such as the formalization of the Feit-Thompson theorem \cite{gonthier2013odd} have required building substantial linear algebra libraries as infrastructure. For example, Aransay and Divas\'on formalized the computation of matrix echelon forms in Isabelle/HOL and proved the correctness of the corresponding algorithm \cite{aransay2016formalisation}. Divas\'on and Thiemann later developed a systematic formalization of the Smith normal form, establishing the correctness of the relevant algorithm in a more general algebraic setting \cite{divason2022formalization}. These developments show that classical results involving matrix transformations, reduction procedures, and canonical forms can be treated in contemporary formal systems, while their proof organization remains closely tied to individual targets.

The target-specific character of these developments is reflected in their proof structure: each formalization is organized around a single algorithm, a single normal form, or an individual theorem \cite{aransay2016formalisation,divason2022formalization}. By contrast, several recurring components of existence proofs for matrix decompositions, such as local transformations, reductions to subproblems, recursion across dimensions, and the recovery of the original decomposition statement from a solved subproblem, have not been organized as a common reusable formal framework. The goal of the present work is to elevate these shared proof components into such a framework and to support systematic formalizations of different decomposition instances. The framework targets decomposition proofs whose recursive organization follows a head-tail reduction pattern; iterative or convergence-based algorithms such as QR iteration are outside its current scope.

Within \mathlib itself, substantial matrix-related infrastructure already exists: determinants, elementary matrix operations, block matrices, the Hermitian spectral theorem (\href{\mathliburl/Mathlib/Analysis/Matrix/Spectrum.lean\#L139}{\breaktt{Matrix.IsHermitian.spectral_theorem}}), and basic triangularity predicates. The present framework does not reprove these results. Instead, it builds upon and imports them where needed, while adding the recursive induction machinery and the shared driver interfaces that \mathlib currently does not provide. The framework's decomposition instances are not intended as thin wrappers around existing \mathlib theorems; they supply decomposition-specific schemas, transformations, reductions, and lifting lemmas beyond the existing library. At the same time, the framework is not currently proposed for upstream inclusion in \mathlib. The primary contribution is the unified proof architecture itself, and the question of which components could eventually be upstreamed is left for future work.

This paper develops a formal framework for existence proofs of matrix decompositions and studies it through a growing library of decomposition instances. The current Lean development no longer treats PLU and QR as isolated endpoints. It uses a common decomposition schema over general index types, square and rectangular matrix universes, and a decomposition driver that turns strategy data into subtype-induction instances. PLU, LU, LDL/Cholesky, QR variants, Gauss rank normal form, Hessenberg reductions, Schur variants, normal spectral decomposition, SVD, bidiagonalization, tridiagonalization, UTV (a unitary--triangular--unitary factorization), Smith normal form, rational canonical form, and Jordan-type forms supply different local algebra while reusing this architecture where their recursive shape permits it. These instances differ along two independent axes: the proof route may be explicit recursive, spectral, or bridge-derived, while the target may be unconditional, conditional, implication-shaped, or enriched with structured trace data (Section~\ref{sec:statement-boundaries} and Table~\ref{tab:quantitative}). Such repeated proof patterns are themselves mathematical data: the common recursive routine becomes a reusable interface, while individual decompositions supply the local algebraic data, structural predicates, and reconstruction lemmas.

The framework formalizes decomposition proofs at the level of certified proof paths rather than executable numerical routines. When an instance supplies explicit local step data, the resulting proof object records how the input is prepared, which subproblem is selected, what measure decreases, how the recursive result is lifted, and how the statement is transported back. In spectral, PID, or canonical-form instances, the same interface still organizes the existence proof, although the recorded object is a logical route through the recursive structure rather than a numerical execution trace. A proof ``path'' here means exactly this logical route --- which local transformation was applied, which slice was taken, and how the result was lifted. This representation accommodates pivoting, Householder or Givens transformations, spectral choices, boundary reductions, and canonical-form bridges. Many constructions are intentionally noncomputable, which keeps the interfaces close to the abstractions available in \mathlib and to textbook mathematics. Executable extraction is a separate layer, addressed for instance by CoqEAL \cite{denes2012refinement}, which refines abstract specifications into executable code. The main contributions of the paper are summarized below.
\begin{enumerate}
  \item We isolate and resolve the central dependent-type obstruction: a recursive slice inhabits a matrix type different from the original, while structural predicates must survive block embeddings and reindexings. Our solution combines reusable decomposition schemas, transformations, reduction methods, and strategies over general index types with reduction-based and subtype-driven well-founded induction.
  \item We package square and rectangular matrix problems in dedicated universe types and provide a common decomposition driver. The driver assembles local strategy data, reduction certificates, and lifting lemmas into reusable induction instances, separating the recursive proof mechanism from the decomposition-specific algebra.
  \item We instantiate the architecture across 17 matrix decomposition and normal-form families. The library covers elimination, orthogonal and unitary reductions, rectangular factorizations, and canonical forms under a shared recursive proof architecture, demonstrating reuse across substantially different algebraic settings (Section~\ref{sec:driver-patterns} and Table~\ref{tab:quantitative}).
  \item We characterize the resulting proof objects as certified proof paths governed by the common interfaces. They record the route from a local transformation through a recursive slice and back to the target statement, while clearly distinguishing formal existence certificates from executable numerical solvers.
\end{enumerate}

These contributions concern both proof reuse and theorem design. The common driver packages a recurring recursive argument; the public predicates determine how closely each formal theorem matches its classical mathematical name. We return to this distinction in Section~\ref{sec:statement-boundaries}. The entire development compiles against Lean~4.25.0 and \mathlib at commit \texttt{0a20708} (November 2025). The repository contains no \texttt{sorry} or \texttt{admit}. The accompanying audit manifest lists 25 exported theorems representing the 17 counted families. Running the repository audit script applies \texttt{\#print axioms} to every listed theorem; the reported dependencies are limited to the three standard axioms \texttt{Classical.choice}, \texttt{propext}, and \texttt{Quot.sound}. The repository\footnote{\url{https://github.com/wl-ma/MatDecompFormal}} includes \texttt{lakefile.toml} and \texttt{lean-toolchain} for reproducible builds. All code listings below reference source paths relative to the repository root.

The remainder of the paper is organized as follows. We recall the recurring proof pattern and the main formalization difficulties in Section~\ref{sec:prelim}. The abstract framework is presented in Section~\ref{sec:framework}, and the technical components are developed in Section~\ref{sec:core-tech}. We then organize the instance library by representative theorem families in Section~\ref{sec:driver-patterns}, explain how the common framework produces decomposition theorems in Section~\ref{sec:plu-case}, and discuss statement boundaries in Section~\ref{sec:statement-boundaries}. A quantitative overview appears in Section~\ref{sec:quantitative}. Reuse, proof paths, statement design, and future extensions are discussed together in Section~\ref{sec:discussion}.

\section{Preliminaries and Motivation}
\label{sec:prelim}
We first recall the recursive proof pattern behind the decompositions considered here, and then explain why this pattern is difficult to represent directly in Lean. Two interacting difficulties drive the formalization. Recursive descent changes matrix dimensions and therefore the matrix types on which the theorem is stated, while reconstruction must lift a solution back through block embeddings and transport its structural properties across reindexings. Accordingly, the first part of this section extracts the recurring transformation--reduction--lifting pattern, and the second identifies the type-changing recursion and block reconstruction obligations that a reusable Lean interface must make explicit in its data.

\subsection{Recurring Proof Pattern}
Many classical existence proofs for matrix decompositions have the same recursive organization, even when the algebraic operations involved are quite different \cite{stewart1998matrix,trefethen2022numerical,golub2013matrix,horn2012matrix,higham2002accuracy}. Let $D_n(A)$ denote the statement that a matrix $A\in \mathcal M_n$ admits the desired decomposition. A typical inductive step has the form
\begin{equation}
\label{eq:recursive-step}
A_1=\mathsf{reduction}(T(A))\in\mathcal M_{n-1},
\qquad
D_{n-1}(A_1)\Longrightarrow D_n(T(A))\Longrightarrow D_n(A).
\end{equation}
Here $T$ is a certified local transformation, $\mathsf{reduction}$ is the map extracting the smaller matrix $A_1$, and the two implications record lifting from a recursive solution followed by transport back along the local transformation. The local transformation may be a row operation, an orthogonal reflection, or a degenerate identity step; the reduced matrix may be a trailing submatrix or a Schur complement. Thus the recursive premise concerns the reduced matrix $\mathsf{reduction}(T(A))$ instead of a decomposition of $A$ itself.

\begin{figure}[t]
\centering
\[
\begin{aligned}
A\in \mathcal M_n
&\xrightarrow{\ T\ }
\widetilde A\in \mathcal M_n
\xrightarrow{\ \mathsf{reduction}\ }
A_1\in \mathcal M_{n-1},\\[0.7em]
D_n(A)
&\xleftarrow{\ \mathsf{transport}\ }
D_n(\widetilde A)
\xleftarrow{\ \mathsf{lift}\ }
D_{n-1}(A_1).
\end{aligned}
\]
  \caption{A recursive decomposition proof as a two-level diagram: local transformation and reduction on matrices, followed by lifting and transport on decomposition statements}
  \label{fig:proof-pattern}
\end{figure}

The matrix maps in \eqref{eq:recursive-step} are displayed separately from the logical steps on decomposition statements in Figure~\ref{fig:proof-pattern}. This distinction is reflected later in the separation between schemas, local transformations, reduction methods, recursive construction, and lifting and transport obligations in the formal framework. The PLU existence proof instantiates this pattern through pivoting and row permutations. If the first column is zero, the trailing block is already the smaller problem. Otherwise, a pivot is moved to the upper-left corner by a row permutation, giving
\[
P_0A=
\begin{pmatrix}
a & u\\
\ell & A_{22}
\end{pmatrix},
\qquad
a\ne 0 .
\]
Applying the corresponding unit lower triangular elimination matrix $L_0^{-1}$ gives
\[
L_0^{-1}P_0A=
\begin{pmatrix}
a & u\\
0 & S
\end{pmatrix}.
\]
Here $S=A_{22}-\ell a^{-1}u$ is the Schur complement. The recursive premise is a PLU statement for the Schur complement $S$, namely $P_1S=L_1U_1$ \cite{stewart1998matrix,zhang2005schur}. Embedding $P_1,L_1,U_1$ into the lower-right block and combining them with the preceding permutation and elimination steps yields factors $P_A,L_A,U_A$ with $P_A A=L_AU_A$. The proof is therefore organized around the passage from local pivoting and elimination to a smaller problem and back to the full PLU statement, rather than around the Schur formula alone.

The square real QR proof gives a structurally different instance of the same organization. A Householder reflection $H_1$ sends the first column of $A\in\mathbb{R}^{n\times n}$ to a vector whose entries below the first are zero \cite{householder1958unitary,trefethen2022numerical,golub2013matrix,bjorck1996numerical}, so that
\[
H_1A=
\begin{pmatrix}
\alpha & w^{T}\\
0 & B
\end{pmatrix}
\]
with $B\in\mathbb{R}^{(n-1)\times(n-1)}$. A recursive QR decomposition of $B$ gives $B=Q_1R_1$. The lifting step forms
\[
\widehat Q=
\begin{pmatrix}
1 & 0\\
0 & Q_1
\end{pmatrix},
\qquad
\widehat R=
\begin{pmatrix}
\alpha & w^{T}\\
0 & R_1
\end{pmatrix},
\qquad
H_1A=\widehat Q\,\widehat R.
\]
Transport along the local Householder transformation yields $A=(H_1^{T}\widehat Q)\widehat R$, where $H_1^{T}\widehat Q$ remains orthogonal. Thus QR has different algebraic content from PLU, while following the same recursive proof organization.

The same pattern appears beyond PLU and QR. LU uses a no-pivot variant of the head-tail reduction; Gauss rank normal form and SVD use rectangular reductions; Hessenberg and tridiagonalization use boundary-style recursive invariants; Smith, rational canonical, and Jordan-type results connect recursive drivers with algebraic bridges. The local algebra varies, but the proof obligations repeatedly have the same shape: make the input ready, extract a smaller problem, prove strict descent, lift the recursive result, and transport the statement back to the original matrix. This repetition is not merely a convenience for implementation. It indicates that the common proof routine can itself be stated and proved as a theorem. In the formal development, the general theorem is represented by reduction-based induction, subtype-driven induction, and driver interfaces that accept decomposition-specific data as parameters. The individual instances are then responsible for their schemas, local transformations, sliceability proofs, lifting lemmas, and transport lemmas. This turns a repeated proof script into a reusable mathematical interface.

\subsection{Two Core Difficulties in Formalization}

The reusable argument encounters two formal obstacles that are largely hidden by ordinary matrix notation. First, recursive descent changes the row and column index types, so the smaller matrix is not an element of the same Lean type as the original matrix. Second, reconstruction must embed the recursive factors into the ambient dimensions and transport structural predicates across the resulting equivalences. Both obstacles originate in the representation of matrices as functions indexed by types rather than as untyped rectangular arrays. We therefore begin with the underlying \mathlib definition: it makes the dependence on the row and column types explicit and explains why a change of dimension must be represented by new types, maps, and equivalences before any decomposition-specific algebra can be reused.

\begin{lstlisting}[caption={Matrix type definition in \mathlib}]
def Matrix (m n : Type*) (α : Type*) := m → n → α
\end{lstlisting}

For finite dimensional matrices, an $n \times n$ matrix is typically written as \breaktt{Matrix (Fin n) (Fin n) R}. This design is well suited to formal linear algebra because dimension information is built into the type itself, and many shape constraints are enforced by typing rather than by separate assumptions \cite{mathlib2020,avigad2020mathematics,gonthier2011pointfree}. At the same time, it makes recursive matrix arguments more delicate. Passing from an $n \times n$ matrix to an $(n-1)\times(n-1)$ submatrix is a transition to a different matrix type, not merely a change of notation. Likewise, block constructions, row and column permutations, and transfers between finite index types must be expressed through explicit reindexing maps and equivalences. Mechanized proofs make both the type-changing reduction and the full-dimensional reconstruction explicit, requiring dedicated lemmas rather than ad hoc arguments \cite{mathlib2020,moura2021lean,coquand1988calculus,deMoura2015leanprover,avigad2015theorem}.

\subsubsection{Recursion Across Dimensions}

In standard mathematical notation, an $n\times n$ matrix can be written in block form, an $(n-1)\times(n-1)$ trailing block or a Schur complement can be isolated, and the same theorem can then be applied to the smaller matrix. This is naturally read as an induction on the size parameter. The formal obstruction is that the smaller matrix is not merely a matrix of smaller size; it inhabits a different type. For instance, the current matrix may have type $\text{\texttt{Matrix (Fin n) (Fin n) R}}$, while a trailing block or Schur complement has type $\text{\texttt{Matrix (Fin (n - 1)) (Fin (n - 1)) R}}$. These types are not definitionally identical, so the recursive call cannot be expressed as if it took place in one fixed matrix type. The Lean sketch below makes this change visible.

\begin{lstlisting}[caption={Recursive type change (sketch)}]
def A  : Matrix (Fin n) (Fin n) R := ...
def A1 : Matrix (Fin (n - 1)) (Fin (n - 1)) R := trailingBlock A
-- The recursive theorem is indexed by the smaller dimension.
have h1 : D (n - 1) A1 := ih (n - 1) A1 ...
\end{lstlisting}

Although $A_1$ belongs to the same mathematical family of matrix problems with a smaller size parameter, its formal type is different. In the current framework this phenomenon is treated more generally: the original and sliced matrices may be indexed by arbitrary finite ordered types, not only by consecutive \texttt{Fin} types. A naive induction over matrices of one fixed type cannot capture the recursive pattern of these proofs. The same issue is clearer in rectangular settings. An $m\times n$ matrix may be reduced to an $(m-1)\times(n-1)$ block, or to a submatrix whose row and column index types change independently. Thus the recursive step changes the size parameters and the matrix type to which the theorem is applied. A fixed-dimension theorem cannot simply have the form

\begin{lstlisting}[caption={Fixed-dimension theorem form (sketch)}]
theorem main_fixed : ∀ A : Matrix (Fin n) (Fin n) R, D n A := ...
\end{lstlisting}
and then recurse internally as if the smaller problem had the same type. The recursive call targets a theorem instance at a different type, and in the rectangular case its row and column types may change independently. We address this by packaging the index types, their finite and order structures, and the matrix into universe objects. The resulting common domain supports a single measure on differently indexed problems, so well-founded induction can compare the original matrix with its slice without pretending that they inhabit one fixed matrix type. The universe construction therefore resolves the typing problem while preserving the index information needed later for block reconstruction and structural transport.

\subsubsection{Block Reconstruction}

A recursive decomposition of the smaller matrix does not by itself prove the original statement. The second difficulty is therefore to lift the recursively obtained factors back to the original dimension and to verify that the resulting enlarged matrices satisfy the desired decomposition statement for the original matrix. In standard mathematical presentations, this step is often compressed into a short block computation. A decomposition is first proved for a trailing block or Schur complement. The smaller factors are then inserted into larger block matrices. The remaining entries are filled using the blocks determined by the local transformation step, and the resulting identity is checked. Mechanization expands this compressed step into explicit constructions and proof obligations. The resulting proof script has the following schematic shape.

\begin{lstlisting}[caption={Schematic lifting proof script (sketch)}]
have hsmall : HasDecomposition smallSchema A1 := ih A1 smaller
let F := liftFactors localData hsmall
have hprop : schema.property F := ...
have heq   : schema.equation A_tilde F := ...
have hmain : HasDecomposition schema A := transport localData ⟨F, hprop, heq⟩
\end{lstlisting}

Even at this level of abstraction, reconstruction separates into three obligations. First, the recursively obtained factors must be embedded into matrices of the original dimension. Second, the enlarged factors must satisfy the reconstruction equation for the original matrix. Third, their structural predicates, such as triangularity, orthogonality, and invertibility, must be preserved after lifting. Reindexing is required mainly for the third obligation and sometimes for the equation proof: the smaller subproblem may use an index type different from the original one, especially after row or column permutations, so preservation facts for the lifted factors must be transported across index equivalences. In Lean pseudocode, this leads to steps of the form
\begin{lstlisting}[caption={Reindexing proof steps (sketch)}]
let A' := Matrix.reindex e1 e2 A
have hU  : IsUpperTriangular U := ...
have hU' : IsUpperTriangular (Matrix.reindex e e U) := ...
\end{lstlisting}
Such steps are usually implicit in informal mathematics, but formally they require explicit lemmas showing that the relevant properties are invariant under reindexing or transported in the intended way. An arbitrary equivalence is insufficient for order-sensitive predicates: triangularity, for example, must be transported along an equivalence that preserves the relevant order comparisons. The reconstruction layer must therefore combine block identities with predicate-specific transport results rather than treating a coordinate change as definitional equality. We first isolate the abstract notions of schemas, transformations, reductions, and strategies in Section~\ref{sec:framework}; we then develop the block-lifting, head-tail equivalence, and reindexing components used to discharge these obligations in Section~\ref{sec:core-tech}.

\section{A Unified Formal Framework}
\label{sec:framework}

We now abstract the proof pattern of Figure~\ref{fig:proof-pattern} into the data required from an individual decomposition proof. The data have three roles. A schema fixes the target proposition $D_n(A)$. A local transformation and a reduction method describe the matrix-level move to a smaller problem. Lifting and transport lemmas connect the recursive premise back to the original statement. A measure records the strict descent used later by the induction principle. With these roles separated, the framework does more than package final existence propositions. It records the recursive construction itself: local normalization, slicing to a smaller matrix, a recursive premise for that slice, and reconstruction in the original dimension. Different decompositions may therefore supply different algebraic ingredients while sharing the same abstract proof skeleton. The induction principle and matrix universe are discussed in Section~\ref{sec:core-tech}.

\subsection{Decomposition Goals}

We use an abstract schema for each decomposition problem. Such a schema records a type $F$ of candidate factors, a predicate $\mathsf{Prop}_F$ expressing the structural constraints on those factors, and a relation $\mathsf{Eq}_F(A,f)$ expressing that the factors reconstruct the input matrix. The associated existence statement has the form $\exists f:F,\ \mathsf{Prop}_F(f)\wedge\mathsf{Eq}_F(A,f)$. For example, in a two-factor schema, the factor type may be a pair $(F_1,F_2)$. The property predicate records the structural constraints on the two factors, and the equation relation may specify $\mathsf{Eq}_F(A,(F_1,F_2))$ as $A=F_1F_2$. The following declarations record this schema.
\begin{lstlisting}[caption={\texttt{DecompositionSchema} and \texttt{HasDecomposition} (\href{\repourl/MatDecompFormal/Abstractions/Schema.lean}{\texttt{Abstractions/Schema.lean}})}]
structure DecompositionSchema (ι κ : Type*) (R : Type*) where
  Factors : Type*
  property : Factors → Prop
  equation : Matrix ι κ R → Factors → Prop

def HasDecomposition {ι κ R} (sch : DecompositionSchema ι κ R) (A : Matrix ι κ R) : Prop :=
  ∃ (factors : sch.Factors), sch.property factors ∧ sch.equation A factors
\end{lstlisting}

This is the canonical schema surface used by the current development. It is indexed directly by row and column types rather than by natural-number dimensions, and it does not impose ring assumptions at the schema level. The definition is minimal: it does not assume any particular factor profile or algebraic identity. The schema does not prescribe the number of factors or their algebraic properties; those choices are supplied by each instance through the factor type, the property predicate, and the equation relation. For a concrete instance, the mathematical work is to choose the factor type, the structural predicate, and the reconstruction equation. Once these choices are made, the associated proposition states that there exists a tuple of factors satisfying the chosen structural predicate and reconstruction equation.

\subsection{Local Transformations}

A local transformation is a certified procedure aimed at a specified target condition. It packages a goal predicate on the current problem, a type of transformation parameters, an application map, and a certified search operation. When the goal is not already satisfied, the search operation produces a parameter whose application makes the goal true, expressed as $\neg\mathsf{Goal}(A)\Longrightarrow \mathsf{Goal}(\mathsf{apply}(\mathsf{find}(A),A))$. We also provide composition operations for local proof steps that naturally decompose into certified stages. The associated result states that the chosen transformation places the matrix in the condition expected by the reduction method, whose specific meaning is supplied by the instance. The corresponding interface has the following form.
\begin{lstlisting}[caption={\texttt{Transformation} interface (\href{\repourl/MatDecompFormal/Abstractions/Transformation.lean}{\texttt{Abstractions/Transformation.lean}})}]
structure Transformation (X : Type*) where
  T : Type*
  Goal : X → Prop
  [decGoal : DecidablePred Goal]
  apply : T → X → X
  find : (x : X) → (h : ¬ Goal x) → T
  find_spec : ∀ (x : X) (h : ¬ Goal x), Goal (apply (find x h) x)
\end{lstlisting}

The certificate supplied by \texttt{find} should be read at the level of this interface rather than as an executable numerical search routine. It is invoked only when the target condition is not already satisfied, and \texttt{find\_spec} certifies that applying the chosen parameter reaches the readiness condition required by the next stage. No recursive slice is selected here, and no decomposition witness is lifted back to the original matrix; those responsibilities belong to the reduction method and the instance-specific proof hooks. This separation also permits a local preparation step to be assembled from several certified stages without forcing every decomposition to expose the same parameter type or numerical implementation. Consequently, transformations with different local algebra can satisfy one abstract contract, while the later strategy determines how that contract is connected to descent and reconstruction.

\subsection{Reduction Methods}

A reduction method records when a matrix is ready to be reduced, how to extract the smaller matrix, and how the original matrix context can be reconstructed after replacing the reduced block by another matrix of the same smaller type. This reconstruction is a structural matrix operation; decomposition-specific lifting is supplied separately by the instance. Abstractly, a reduction method consists of a predicate $\mathsf{IsSliceable}(A)$, a slicing map $A\mapsto\mathsf{slice}(A)$ defined under that predicate, and a reconstruction map $(A,h_A,B) \mapsto \mathsf{reconstruct}(A,h_A,B)$, where $B:\mathcal M_{\mathrm{slice}}$. Its consistency law states that replacing the slice by itself reconstructs the original matrix, $\mathsf{reconstruct}(A,h_A,\mathsf{slice}(A,h_A))=A$. The argument \texttt{slice\_sol} is a replacement matrix for the reduced block, not a decomposition proof or a tuple of factors. Concrete reusable reductions, including lower-right submatrix, zero-column, and Schur complement reductions, are discussed in Section~\ref{sec:core-tech}. These data and the reconstruction law are recorded by the following interface.

\begin{lstlisting}[caption={\texttt{ReductionMethod} interface (\href{\repourl/MatDecompFormal/Abstractions/ReductionMethod.lean}{\texttt{Abstractions/ReductionMethod.lean}})}]
structure ReductionMethod (ι κ ιs κs : Type*) (R : Type*) where
  IsSliceable : Matrix ι κ R → Prop
  slice : (A : Matrix ι κ R) → (hA : IsSliceable A) → Matrix ιs κs R
  reconstruct : (A : Matrix ι κ R) → (hA : IsSliceable A) → (slice_sol : Matrix ιs κs R) → Matrix ι κ R
  reconstruct_slice_eq : ∀ (A : Matrix ι κ R) (hA : IsSliceable A), reconstruct A hA (slice A hA) = A
\end{lstlisting}

\subsection{Strategies and Branches}
A strategy connects the local transformation with the reduction method. It records that the transformation goal equals the reduction's sliceability condition, and it equips the current problem and the slice with measures for recursive progress. This compatibility is an equality of predicates, not only an implication. The equality allows the certificate produced by the transformation to be consumed directly by the slice operation, without an additional conversion lemma in every instance. The key progress conditions are shown in \eqref{eq:strategy-progress}: normalization does not increase complexity, while slicing gives the strict descent used for reachability in a form reusable across instances.
\begin{equation}
\label{eq:strategy-progress}
\mu(\mathsf{apply}(t,A)) \le \mu(A),
\qquad
\mu_{\mathrm{slice}}(\mathsf{slice}(B,h_B)) < \mu(B).
\end{equation}
\begin{lstlisting}[caption={\texttt{ReductionStrategy} interface (\href{\repourl/MatDecompFormal/Abstractions/Strategy.lean}{\texttt{Abstractions/Strategy.lean}})}]
structure ReductionStrategy (ι κ ιs κs : Type*) (R : Type*) where
  transform : Transformation (Matrix ι κ R)
  reduction : ReductionMethod ι κ ιs κs R
  goal_is_sliceable : transform.Goal = reduction.IsSliceable
  μ : Matrix ι κ R → ℕ
  μ_slice : Matrix ιs κs R → ℕ
  μ_mono : ∀ (A : Matrix ι κ R) (t : transform.T), μ (transform.apply t A) ≤ μ A
  slice_progress : ∀ (A : Matrix ι κ R) (hA : reduction.IsSliceable A), μ_slice (reduction.slice A hA) < μ A
\end{lstlisting}

From a strategy, we derive a local reachability principle stating that every non-base matrix can be moved, possibly unchanged, into a sliceable form whose recursive slice is strictly smaller. This principle is encapsulated by \href{\repourl/MatDecompFormal/Abstractions/Strategy.lean}{\texttt{ReductionStrategy.mk\_reach}} and supplies the reachability hypothesis required by the induction principle. Branching is handled compositionally. For two reduction methods with the same source and slice dimensions, the reduction-level fallback combinator used here, \texttt{ReductionMethod.try\_else}, forms a left-biased branch whose sliceability predicate is the disjunction $\mathsf{IsSliceable}(A) \equiv \mathsf{IsSliceable}_1(A)\lor\mathsf{IsSliceable}_2(A)$. This supports proof procedures whose reduction step has several admissible branches.

The workflow mirrors Figure~\ref{fig:proof-pattern}. A schema fixes the target proposition, a strategy provides a certified route to a sliceable matrix and a smaller recursive problem, and instance-specific lifting and transport lemmas connect the recursive premise back to the original statement. Once these components are supplied, the local step yields the final existence theorem through the induction principle discussed in Section~\ref{sec:core-tech}. This separation supports reuse because theorem assembly depends on the abstract components above rather than on the specific algebra used to construct them, while the instance remains responsible for the algebraic content of its witnesses and factor equations.

\subsection{The Decomposition Driver}

The current framework also contains a driver layer that packages the recurring proof routine into reusable square and rectangular assembly mechanisms. At this level, a decomposition instance no longer constructs the subtype-induction instance by hand. It supplies a strategy core, proof hooks for transport and lifting, and a base case. The driver converts these data into a universe-level induction instance. Separate square and rectangular constructors retain the difference between one changing index type and two independently changing index types, while presenting the same assembly boundary to the final proof. The declarations below show the data records and the constructors generated from them.
\begin{lstlisting}[caption={Strategy data and driver constructors (\href{\repourl/MatDecompFormal/Framework/DecompositionDriver.lean}{\texttt{Framework/DecompositionDriver.lean}})}]
structure SquareStrategyData (R : Type*) (P : SquareUniverse R → Prop) where
  core : SquareStrategyCore R
  proofData : SquareStrategyProofData R P core

structure RectStrategyData (R : Type*) (P : RectUniverse R → Prop) where
  core : RectStrategyCore R
  proofData : RectStrategyProofData R P core

noncomputable def mkSquareSubtypeInductionInstanceFromStrategy ...
noncomputable def mkRectSubtypeInductionInstanceFromStrategy ...
\end{lstlisting}

The square and rectangular cores specify how slice index types are chosen from the ambient index types, how the local strategy is obtained, and how the measure agrees with the cardinality-based universe measure. The proof data then supplies the two decomposition-specific hooks: transport along the local transformation relation and lifting from the recursive slice. Thus the driver is the code-level form of the general statement described above. It turns repeated proof routines into one reusable theorem interface, while leaving the algebra of pivoting, orthogonal transformations, spectral choices, or canonical-form bridges entirely to the individual instances themselves at each concrete application.

\section{Core Technical Components}
\label{sec:core-tech}

We now describe the technical components that support the framework: matrix universes for cross-dimensional recursion, induction principles, block lifting, reindexing, transport, and reusable reduction components. The universe and induction layers address recursion across changing matrix types by placing differently indexed matrices in a common measured domain. The lifting, reindexing, and transport layers reconstruct full-dimensional statements from reduced subproblems while preserving their structural predicates. The section follows this division: it first develops the universe and induction machinery, then the common reconstruction tools and reduction components that discharge the two formal difficulties isolated above in Section~\ref{sec:prelim} for every decomposition instance.

\subsection{The Matrix Universe}

Recursive matrix arguments change index types, so they cannot be expressed inside one fixed matrix type. We therefore introduce matrix universes whose elements package the row and column index types, the required finite and order structure, and the matrix itself. (The term ``universe'' here refers to a packaging type for parameterized matrix problems, not to the Lean type-theoretic universe hierarchy \texttt{Type~u}.) Thus a theorem can be stated over one type of indexed matrix problems, rather than separately over each concrete matrix type. The package also exposes the cardinality measure needed to compare an ambient problem with a slice whose indices have different Lean types.

This representation is needed because Lean does not identify different presentations of the same finite shape by definitional equality. For example, \texttt{Fin (n+1)} and the block index type \texttt{Fin 1} $\oplus$ \texttt{Fin n} are canonically equivalent but not definitionally equal. Likewise, \texttt{Matrix (Fin (n+1)) (Fin (n+1)) R} is not the same type as a block matrix indexed by \texttt{Fin 1} $\oplus$ \texttt{Fin n}, nor as the smaller trailing matrix type \texttt{Matrix (Fin n) (Fin n) R}. In CIC, these types are not identified definitionally just because they represent the same finite shape in informal mathematics. The matrix universe still requires explicit equivalences. Its purpose is to provide the common domain in which the recursive call, the slice, and the original matrix can be related by explicit maps and proof obligations. The square universe packages one index type for both rows and columns, while the rectangular universe packages separate row and column index types. The framework also isolates positive-dimensional regions as recursive subtypes: in the square case the index cardinality is positive, while in the rectangular case both row and column cardinalities are positive. Zero-dimensional or empty-index objects form the base cases, and positive objects are those on which recursive slicing is meaningful. The corresponding Lean definitions have the following shape.

\begin{lstlisting}[caption={\texttt{RectUniverse} and \texttt{SquareUniverse} (\href{\repourl/MatDecompFormal/Framework/Universe.lean}{\texttt{Framework/Universe.lean}})}]
structure RectUniverse (R : Type*) where
  ι : Type*
  [fintype_ι : Fintype ι]
  [decEq_ι : DecidableEq ι]
  [linOrder_ι : LinearOrder ι]
  κ : Type*
  [fintype_κ : Fintype κ]
  [decEq_κ : DecidableEq κ]
  [linOrder_κ : LinearOrder κ]
  A : Matrix ι κ R

structure SquareUniverse (R : Type*) where
  ι : Type*
  [fintype_ι : Fintype ι]
  [decEq_ι : DecidableEq ι]
  [linOrder_ι : LinearOrder ι]
  A : Matrix ι ι R
\end{lstlisting}

\subsection{Induction Principle}

We use the induction principle through three layers of formal packaging. The first is a reduction-based well-founded induction theorem over an abstract set $X$. It only assumes a base predicate, a reducibility predicate, a reduction map, a lift step, and a transport relation. The second layer is the subtype-driven specialization used in the matrix applications: recursion is performed only on a recursive subtype, such as positive-dimensional square matrices, while the final predicate is stated on the full matrix universe. The third layer is the decomposition driver, which adapts concrete square or rectangular strategy families to the subtype formulation in a uniform way.

The induction principles below are not standard applications of \texttt{WellFounded.fix} from \mathlib. The standard combinator requires that the recursive call target the same type as the current element. In matrix decomposition proofs, the current matrix has type \texttt{Matrix (Fin~n) (Fin~n)~R} while the recursive slice has type \texttt{Matrix (Fin~(n-1)) (Fin~(n-1))~R} --- a different type. Furthermore, after slicing, the recursive result must be lifted through block embeddings that pass through \texttt{Fin~1~$\oplus$~Fin~n}, which is equivalent to but not definitionally equal to \texttt{Fin~(n+1)}. Order-sensitive predicates such as triangularity require order-preserving equivalences for transport, not just bijections. The combination of cross-type recursion, universe packaging, and predicate transport along order-compatible reindexings is what the present induction principles formalize. The following formal statement gives the resulting reduction-based induction principle.
\begin{lstlisting}[caption={Reduction-based induction (\href{\repourl/MatDecompFormal/Framework/Induction.lean}{\texttt{Framework/Induction.lean}})}]
theorem induction_by_reduction
    {rel : X → X → Prop} (hwf : WellFounded rel)
    (BaseSet : X → Prop)
    {r : X → X → Prop} {P : X → Prop}
    (h_trans : Transport r P)
    (IsReducible : X → Prop)
    (decompose : ∀ {x : X}, IsReducible x → X)
    (reconstruct : ∀ {x : X} (hx : IsReducible x), P (decompose hx) → P x)
    (prove_on_base : ∀ {x : X}, BaseSet x → P x)
    (reach_from_non_base : ∀ {x : X}, ¬ BaseSet x → ∃ y, ∃ (hy : IsReducible y), r y x ∧ rel (decompose hy) x) : ∀ (x : X), P x := by
\end{lstlisting}
\begin{theorem}[Reduction-based induction]
Let $X$ be a set equipped with a well-founded strict relation $\prec$. This
relation is the strict descent relation used for induction. Let
$B\subseteq X$ be the base predicate, let
$R\subseteq X$ be the predicate of reducible objects, and let
$\delta:\{y\in X\mid R(y)\}\to X$ be the reduction map, sending a reducible
element to its recursive slice or smaller element. Let $\sim$ be a local
transport relation on $X$, not assumed to be an equivalence, and let
$P:X\to\mathrm{Prop}$ be the target property. Suppose that
\begin{enumerate}
\item $P(x)$ holds for every $x\in B$;
\item for every $x\notin B$, there are $y\in X$ and $h_y:R(y)$ such that
$y\sim x$ and $\delta(y,h_y)\prec x$;
\item for every $y\in X$ and every $h_y:R(y)$,
$P(\delta(y,h_y))$ implies $P(y)$;
\item for all $x,y\in X$, $y\sim x$ and $P(y)$ imply $P(x)$.
\end{enumerate}
Then $P(x)$ holds for every $x\in X$.
\end{theorem}
\begin{proof}
We argue by well-founded induction on $x$. If $x\in B$, the base hypothesis
gives $P(x)$. In the non-base case, choose $y$ and $h_y:R(y)$ as in the
reachability hypothesis. Since $\delta(y,h_y)\prec x$, the induction
hypothesis yields $P(\delta(y,h_y))$. The lift hypothesis gives $P(y)$, and
the transport hypothesis, applied to $y\sim x$, gives $P(x)$. This proves the
induction step, and well-foundedness gives the result.
\end{proof}

A particularly important specialization is the subtype-driven induction principle \href{\repourl/MatDecompFormal/Framework/Induction.lean}{\breaktt{induction_on_subtype'}}. It separates the predicate on the full set $X$ from the recursive predicate on a distinguished class $S$, such as the positive-dimensional square matrices. The recursive call is made only for a slice from $S$, whose measure is strictly smaller, while compatibility between the two predicates transfers the result back to $X$. Objects outside the recursive subtype are discharged by the universe-level base condition. This is the subtype version of the reduction-based induction principle above; its mathematical content is stated after the Lean interface.
\begin{lstlisting}[caption={Subtype-driven induction (\href{\repourl/MatDecompFormal/Framework/Induction.lean}{\texttt{Framework/Induction.lean}})}]
theorem induction_on_subtype' (SubX : Type*) (toX : SubX → X)
    (μ : X → α) (relα : α → α → Prop) (hwf : WellFounded relα)
    (P : X → Prop)
    (P_sub : SubX → Prop)
    (P_compat : ∀ (x_sub : SubX), P_sub x_sub ↔ P (toX x_sub))
    (r_sub : SubX → SubX → Prop)
    (IsSliceable_sub : SubX → Prop)
    (slice_sub : ∀ (x_sub : SubX), IsSliceable_sub x_sub → X)
    (transport_sub : ∀ {x_sub y_sub}, r_sub y_sub x_sub → P_sub y_sub → P_sub x_sub)
    (lift_from_slice_sub : ∀ (x_sub : SubX) (hx : IsSliceable_sub x_sub), P (slice_sub x_sub hx) → P_sub x_sub)
    (BaseSet : X → Prop)
    (reach_sub : ∀ (x_sub : SubX), ¬ BaseSet (toX x_sub) → Σ' (y_sub : SubX), Σ' (hy : IsSliceable_sub y_sub), r_sub y_sub x_sub ∧ relα (μ (slice_sub y_sub hy)) (μ (toX x_sub)))
    (base_univ : ∀ (x : X), (∀ (x_sub : SubX), toX x_sub ≠ x) ∨ BaseSet x → P x) : ∀ (x : X), P x := by ...
\end{lstlisting}
\begin{theorem}[Subtype-driven induction]
Let $X$ be a set of elements, and let $S$ be a distinguished class of recursive
elements with inclusion $i:S\to X$. Let $\alpha$ carry a well-founded strict
order $<$, and let $\mu:X\to\alpha$ be a measure. Write $z\prec_\mu x$ when
$\mu(z)<\mu(x)$; this is the strict descent relation used for recursion. Let
$P:X\to\mathrm{Prop}$ be the target property, and let $P_S:S\to\mathrm{Prop}$
be the corresponding formulation on the recursive class, satisfying
$P_S(s)\leftrightarrow P(i(s))$.

Let $B\subseteq X$ be the base or nonrecursive predicate. On $S$, let
$R_S\subseteq S$ be the sliceability predicate, let
$\sigma:\{s\in S\mid R_S(s)\}\to X$ be the slice map, and let $\sim_S$ be the
local transport relation. Assume that
\begin{enumerate}
\item $P(x)$ holds whenever $x\notin \mathrm{im}(i)$ or $x\in B$;
\item if $s\in S$ and $i(s)\notin B$, then there are $t\in S$ and
$h_t:R_S(t)$ such that $t\sim_S s$ and $\sigma(t,h_t)\prec_\mu i(s)$;
\item for every $t\in S$ and every $h_t:R_S(t)$,
$P(\sigma(t,h_t))$ implies $P_S(t)$;
\item for all $s,t\in S$, $t\sim_S s$ and $P_S(t)$ imply $P_S(s)$.
\end{enumerate}
Then $P(x)$ holds for every $x\in X$.
\end{theorem}

\begin{proof}
The relation $\prec_\mu$ is well-founded because it is induced by the
well-founded order on $\alpha$. Apply the reduction-based induction principle
with this relation. If $x\notin \mathrm{im}(i)$ or $x\in B$, then $P(x)$
follows from the first hypothesis. It remains to consider $x=i(s)$ with
$s\in S$ and $i(s)\notin B$. Choose $t$ and $h_t:R_S(t)$ from the reachability
hypothesis. Since $\sigma(t,h_t)\prec_\mu i(s)$, the induction hypothesis gives
$P(\sigma(t,h_t))$. The lift hypothesis gives $P_S(t)$, and the transport
hypothesis gives $P_S(s)$. By compatibility of $P_S$ with $P$, this is
equivalent to $P(i(s))$. Hence the reduction-based induction principle proves
$P$ on all of $X$.
\end{proof}

With these layers separated, the instance supplies the local reachability, lifting, and transport lemmas, while the induction principle supplies only the well-founded recursion core. The driver layer adds the final packaging step here. It converts a local strategy family over arbitrary finite ordered index types into square or rectangular slice data over the matrix universes, and it exposes a theorem of the form \texttt{prove\_for\_matrix}. This packaging replaces the earlier need to write each universe-level subtype instance directly and gives all later instances a stable matrix-level entry point independent of their internal slice types, dimensions, factor equations, and algebraic hypotheses in each case.

\begin{lstlisting}[caption={Matrix-level theorem from induction instance (\href{\repourl/MatDecompFormal/Framework/UniverseDecomposition.lean}{\texttt{Framework/UniverseDecomposition.lean}})}]
theorem SquareSubtypeInductionInstance.prove_for_matrix
    (inst : SquareSubtypeInductionInstance R)
    {ι : Type*} [Fintype ι] [DecidableEq ι] [LinearOrder ι]
    (A : Matrix ι ι R) : inst.P (SquareUniverse.ofMatrix A)

theorem RectSubtypeInductionInstance.prove_for_matrix
    (inst : RectSubtypeInductionInstance R)
    {ι κ : Type*} [Fintype ι] [DecidableEq ι] [LinearOrder ι]
    [Fintype κ] [DecidableEq κ] [LinearOrder κ]
    (A : Matrix ι κ R) : inst.P (RectUniverse.ofMatrix A)
\end{lstlisting}

\subsection{Block Lifting}

After a recursive subproblem is solved, its factors must be lifted to a decomposition of the original matrix. The block-lifting lemmas separate common block calculation from decomposition-specific algebra. They operate on block matrices
\[
A =
\begin{pmatrix}
A_{11} & A_{12}\\
A_{21} & A_{22}
\end{pmatrix},
\]
where the recursive factorization is already known for the trailing block or for a derived slice such as a Schur complement. In the three-factor lifting theorems, Schur-type reductions lift a factorization of $A_{22} - A_{21}A_{11}^{-1}A_{12}$ to one of $A$, while zero-column reductions lift a factorization of $A_{22}$ to one of
\[
\begin{pmatrix}
0 & A_{12}\\
0 & A_{22}
\end{pmatrix}.
\]
These principles assume a slice factorization and the relevant stability facts under reindexing, and they produce lifted factors satisfying the full matrix equation. The display below suppresses the concrete block expressions: the instance supplies the factor predicates and proves that the lifted factors satisfy them. For PLU these predicates require permutation, unit lower triangularity, and upper triangularity; for QR they require orthogonality and upper triangularity. The generic principles therefore remove repeated block algebra without themselves proving a PLU or QR lifting theorem. A Schur-type three-factor step has the following schematic form.
\begin{lstlisting}[caption={Schur-type three-factor lifting (sketch)}]
lift_schur_pattern
  inputs:
    A : Matrix (Fin (k + 1)) (Fin (k + 1)) R
    a proof that A₀₀ is invertible
    factors subF₁, subF₂, subF₃ for the Schur slice
    proofs that the factor predicates survive the block reindexings
    equation: subF₁ * (the Schur slice of A) = subF₂ * subF₃
  output:
    lifted factors F₁, F₂, F₃ with F₁ * A = F₂ * F₃
\end{lstlisting}

\subsection{Reindexing and Transport}

Reindexing is used when block arguments replace an ambient finite ordered index type by a head-tail presentation. Both square and rectangular drivers use this operation to identify the recursive slice with a matrix over the tail index type. These uses require transport lemmas for structural properties under reindexing. For predicates depending only on entries, an equivalence of index types is enough. For order-sensitive predicates such as upper and lower triangularity, the equivalence must also preserve order. Triangularity is a predicate on entries together with ordered index comparisons, not merely a predicate on entries. This is why the head-tail equivalences are accompanied by order-compatible variants. The foundational equivalence splits a nonempty finite ordered index type into its head element and its tail subtype; the following Lean definitions record this layer.

\begin{lstlisting}[caption={Head-tail equivalences (\href{\repourl/MatDecompFormal/Framework/HeadTail.lean}{\texttt{Framework/HeadTail.lean}})}]
noncomputable def headTailEquiv [Fintype ι] [LinearOrder ι] [Nonempty ι] : ι ≃ Unit ⊕ {a : ι // a ≠ headElem}

noncomputable def headTailLexEquiv [Fintype ι] [LinearOrder ι] [Nonempty ι] : ι ≃ Unit ⊕ₗ {a : ι // a ≠ headElem}
\end{lstlisting}

The second equivalence maps into the lexicographic sum type $\oplus_\ell$; strict monotonicity is proved separately as \texttt{headTailLexEquiv\_strictMono}. These equivalences underlie block extraction identities. Schematically, the tail-tail submatrix is the bottom-right block of a head-tail reindexing. For order-sensitive predicates, the block equivalence uses the order-compatible variant \texttt{sumToLexEquiv} with a strict monotonicity proof, so that such predicates persist under order-preserving equivalences. The same reindexing layer also transports permutation matrices, diagonal entries, lower triangularity, and unit lower triangularity. Its monotonicity hypothesis transports comparisons such as $i>j$ and prevents an arbitrary permutation from destroying triangular structure; the resulting lemmas therefore connect informal coordinate changes with the type-theoretic requirement that predicate preservation be proved explicitly. Upper triangularity, for example, is stable under the reindexing stated below.

\begin{lstlisting}[caption={Block extraction identity (\href{\repourl/MatDecompFormal/Framework/Reindex.lean}{\texttt{Framework/Reindex.lean}})}]
lemma submatrix_inr_inr_eq_toBlocks₂₂
    (er : ι ≃ ι₁ ⊕ ι₂) (ec : κ ≃ κ₁ ⊕ κ₂) (A : Matrix ι κ R) :
    A.submatrix (fun i => er.symm (Sum.inr i)) (fun j => ec.symm (Sum.inr j)) = (Matrix.reindex er ec A).toBlocks₂₂ := by ...
\end{lstlisting}

\begin{lstlisting}[caption={Triangularity under reindexing (\href{\repourl/MatDecompFormal/Components/Properties/Reindex.lean}{\texttt{Components/Properties/Reindex.lean}})}]
lemma isUpperTriangular_reindex (e : ι ≃ ι') (h_mono : StrictMono e)
    (A : Matrix ι ι R) : IsUpperTriangular A ↔ IsUpperTriangular (A.reindex e e) := by ...
\end{lstlisting}
\subsection{Representative Reduction Mechanisms}

The reusable reduction layer is deliberately smaller than the collection of recursive moves appearing in individual decompositions. We compare three representative cases according to the slice they extract, the context they reconstruct, and the algebra that remains specific to a theorem family. Lower-right submatrix extraction and zero-column reconstruction admit definitions that are independent of the eventual factor equation, so they are packaged as shared components. Schur-complement reduction has the same abstract role in recursive descent, but its elimination coefficients and reconstruction identities remain tied to the surrounding instance in the current development. The comparison therefore separates a common reduction interface from the stronger claim that every concrete recursive move has already been factored into a single reusable implementation.

Among the shared components, \texttt{SubmatrixMethod} is the default mechanism when local normalization allows recursion on the lower-right block $A \mapsto A_{22}$. The interface below is independent of any particular decomposition theorem: it packages the row and column equivalences, the sliceability predicate, extraction of the lower-right block, and reconstruction of the surrounding block context. It deliberately stops before decomposition-specific lifting, so PLU, QR, and rectangular instances can reuse the same slice operation while proving different factor equations and structural predicates without changing its definition. This boundary makes lower-right extraction reusable without requiring the reduction component to know which factors will eventually be reconstructed from the recursive witness.

\begin{lstlisting}[caption={\texttt{SubmatrixMethod} (\href{\repourl/MatDecompFormal/Components/Reductions/Submatrix.lean}{\texttt{Components/Reductions/Submatrix.lean}})}]
noncomputable def SubmatrixMethod
    {ι κ ι₁ ι₂ κ₁ κ₂ R : Type*}
    (er : ι ≃ ι₁ ⊕ ι₂) (ec : κ ≃ κ₁ ⊕ κ₂)
    (IsSliceable_def : Matrix ι κ R → Prop) :
    ReductionMethod ι κ ι₂ κ₂ R where
  IsSliceable := IsSliceable_def
  slice := fun A hA => A.submatrix ... ...
  reconstruct := fun A hA slice_sol => ...
  reconstruct_slice_eq := by ...
\end{lstlisting}

A second shared mechanism handles the degenerate case in which the first column has already vanished. The slice is again the trailing block, but the reconstruction is specialized to the block form displayed below. Unlike a Schur-complement step, no invertible pivot or elimination coefficient is needed; the zero first column is retained while the recursive solution replaces only the lower-right block. This distinction makes the method useful as a fallback branch in proofs such as PLU, where the reduction selected for a matrix depends on whether pivoting is necessary. The subsequent \texttt{ZeroColumnMethod} declaration records this reusable slice and reconstruction operation:
\[
\begin{pmatrix}
0 & A_{12}\\
0 & A_{22}
\end{pmatrix}.
\]
\begin{lstlisting}[caption={\texttt{ZeroColumnMethod} (\href{\repourl/MatDecompFormal/Components/Reductions/ZeroColumn.lean}{\texttt{Components/Reductions/ZeroColumn.lean}})}]
noncomputable def ZeroColumnMethod
    {ι κ ι₁ ι₂ κ₁ κ₂ R : Type*} [Zero R] [Unique κ₁]
    (er : ι ≃ ι₁ ⊕ ι₂) (ec : κ ≃ κ₁ ⊕ κ₂) :
    ReductionMethod ι κ ι₂ κ₂ R where
  IsSliceable := fun A => ...
  slice := fun A hA => ...
  reconstruct := fun A hA slice_sol => ...
  reconstruct_slice_eq := by ...
\end{lstlisting}

The third case marks the boundary of the shared reduction library. For square matrices over a field, a common recursive step replaces the trailing block with the Schur complement $A_{22} - A_{21}A_{11}^{-1}A_{12}$ when the leading entry is invertible \cite{zhang2005schur}. Unlike the two reductions above, this operation is not factored into a standalone reusable component; instead, each instance that uses it defines an instance-specific reduction method. For example, the PLU instance defines \texttt{pluPivotSchurReduction} as a \texttt{ReductionMethod} whose slice computes the Schur complement and whose reconstruction restores the original block context. The same pattern appears in LU. Extracting a generic \texttt{SchurMethod} would require parameterizing over the specific elimination and reconstruction algebra, which varies across instances; the current design therefore keeps Schur-complement reductions local while reusing \texttt{SubmatrixMethod} and \texttt{ZeroColumnMethod}. These examples show which parts of recursive descent can be packaged independently of a decomposition theorem and which remain tied to instance-specific elimination algebra. We next explain how such reduction data is combined with the remaining proof obligations to instantiate the framework.

\subsection{Instantiating the Framework}

To instantiate the framework, an instance supplies a schema, a local transformation, a reduction method, and a measure. It also supplies the connecting proofs, including base cases, reachability, strict decrease, lifting, transport, and reindexing or order-transport lemmas when the structural predicates require them. The schema alone does not prove a decomposition. We use the induction principle and the decomposition driver to assemble these supplied fields into a theorem over the matrix universe. The resulting record makes each proof obligation explicit and separates generic recursion from local algebra. In the square case, the induction instance has the following shape shown below.
\begin{lstlisting}[caption={\texttt{SubtypeInductionInstance} (\href{\repourl/MatDecompFormal/Framework/UniverseDecomposition.lean}{\texttt{Framework/UniverseDecomposition.lean}})}]
structure SubtypeInductionInstance (X SubX : Type*) (toX : SubX → X) where
  μ : X → Nat
  μ_base : Nat
  P : X → Prop
  P_sub : SubX → Prop
  P_compat : ∀ (x_sub : SubX), P_sub x_sub ↔ P (toX x_sub)
  r_sub : SubX → SubX → Prop
  IsSliceable_sub : SubX → Prop
  slice_sub : ∀ (x_sub : SubX), IsSliceable_sub x_sub → X
  transport_sub : ∀ {x_sub y_sub}, r_sub y_sub x_sub → P_sub y_sub → P_sub x_sub
  lift_from_slice_sub : ∀ (x_sub : SubX) (hx : IsSliceable_sub x_sub), P (slice_sub x_sub hx) → P_sub x_sub
  reach_sub : ∀ (x_sub : SubX), μ (toX x_sub) > μ_base → Σ' (y_sub : SubX), Σ' (hy : IsSliceable_sub y_sub), r_sub y_sub x_sub ∧ μ (slice_sub y_sub hy) < μ (toX x_sub)
  base_univ : ∀ (x : X), (∀ (x_sub : SubX), toX x_sub ≠ x) ∨ μ x ≤ μ_base → P x
\end{lstlisting}

The theorem over the matrix universe is obtained by a single application of the subtype-induction principle. Once the instance-specific proof ingredients have been supplied through this interface, the subtype-induction principle yields the theorem $\forall x:X,\ P(x)$. At this point, decomposition-specific algebra has been separated from the well-founded recursive argument. The instance proves the decomposition-specific lemmas, the required fields are recorded in the abstract interface, and the recursive argument over the matrix universe is supplied by the induction principle. A matrix-level wrapper then specializes the universe theorem uniformly to arbitrary finite ordered row and column index types in the public statement.

\section{Representative Driver Patterns}
\label{sec:driver-patterns}

The instance library is most informative when organized by the interface pattern through which each theorem enters the subtype-induction driver, rather than as a catalogue of final formulas. This section identifies which parts the common driver handles, which remain local to each decomposition family, and where statement strength changes the meaning of the final theorem. Figure~\ref{fig:induction-tree} summarizes this assembly from top to bottom. The universe layer records the ambient type $X$, recursive subtype $\mathit{SubX}$, measure $\mu$, and structural tags; the target layer supplies the predicate $P:X\to\mathsf{Prop}$ and subtype compatibility; the strategy layer supplies local steps, slices, reachability, and descent; and the proof layer supplies base, lift, and transport obligations. Driver assembly produces a \emph{Subtype Induction Instance}, from which \texttt{prove\_for\_matrix} exposes a matrix-level theorem. Section~\ref{sec:plu-case} traces this route in detail, while the present section compares the resulting theorem families, their factors and equations, and their statement boundaries.

\begin{figure}[t]
\centering
\includegraphics[width=\textwidth]{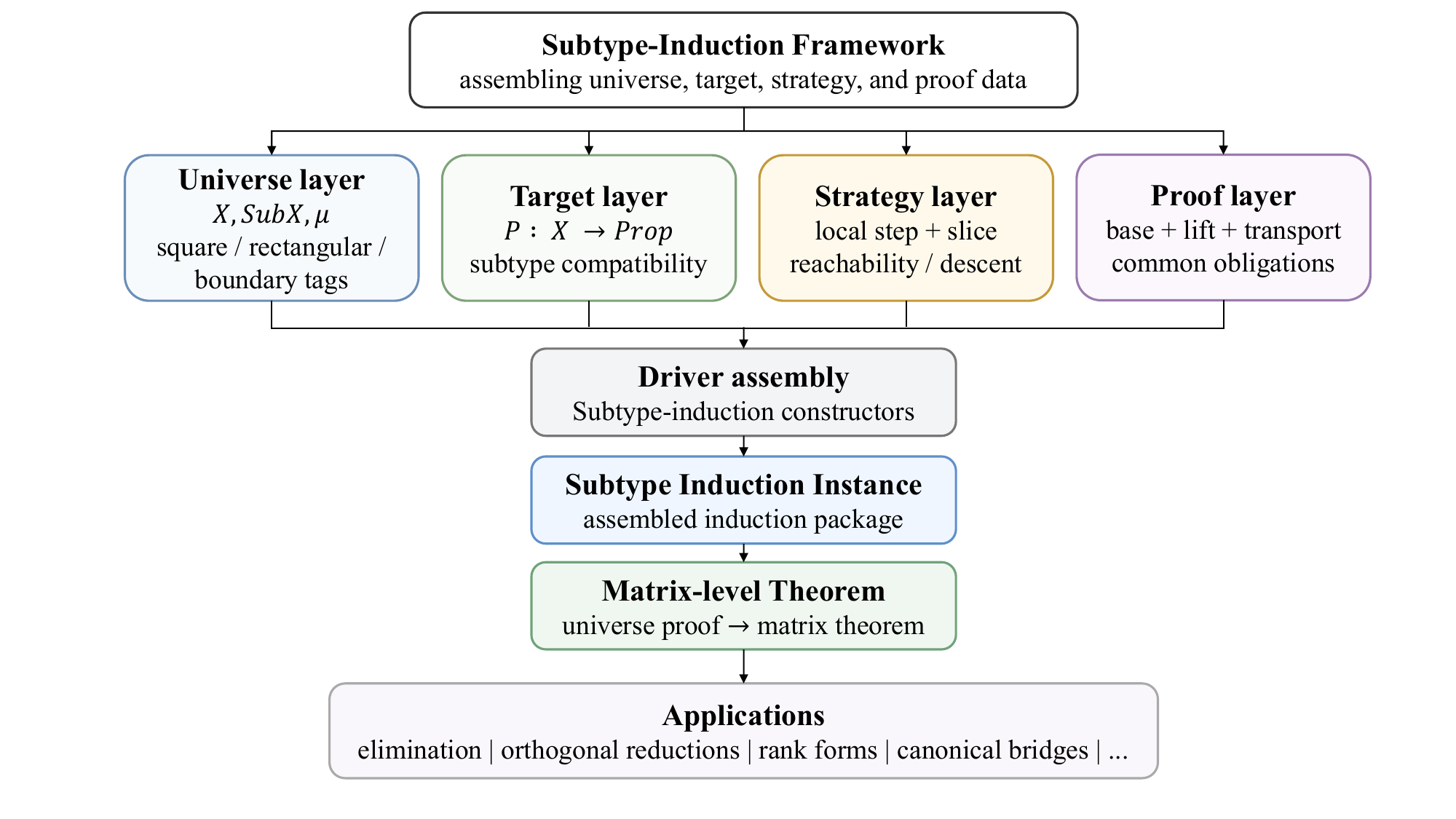}
\caption{\emph{Subtype-Induction Framework}. Universe, target, strategy, and proof data are assembled through the driver into a subtype-induction instance. The resulting universe proof is converted into matrix-level theorems, with applications including elimination, orthogonal reductions, rank forms, and canonical bridges}
\label{fig:induction-tree}
\end{figure}

\subsection{Square Elimination and Triangularization}

The square elimination branch is the clearest place where the common framework meets textbook head-tail recursion. A local transformation prepares the first row and column, a tail or Schur-type slice is extracted, the recursive theorem is applied to a smaller square matrix, and the resulting witness is lifted back through block structure. What changes from theorem to theorem is not the presence of a square recursive pattern, but the mathematical claim expressed at the end of that pattern. PLU and LU supply different local algebra to this same skeleton. PLU branches between a zero-column case and a pivot/Schur branch, while LU uses a no-pivot readiness predicate and hides that internal predicate behind public theorems such as \texttt{exists\_lu}. The square driver is therefore reused at two levels: first for unconditional branching statements such as PLU, and second for conditional square theorems whose public assumptions are proved equivalent to the internal recursive readiness predicate.
\begin{lstlisting}[caption={LU schema and existence theorem (\href{\repourl/MatDecompFormal/Instances/LU/Details.lean}{\texttt{Instances/LU/Details.lean}}; \href{\repourl/MatDecompFormal/Instances/LU/Existence.lean}{\texttt{Instances/LU/Existence.lean}})}]
def LU_Schema : DecompositionSchema ι ι R where
  Factors := Matrix ι ι R × Matrix ι ι R
  property := fun (L, U) => IsUnitLowerTriangular L ∧ IsUpperTriangular U
  equation := fun A (L, U) => A = L * U

theorem exists_lu (A : Matrix ι ι R) (hA : HasNoZeroLUPivots A) : HasLU A
\end{lstlisting}

This listing shows the first kind of difference that matters in the square branch: the theorem surface itself. LU proves a conditional factorization theorem $A=LU$ over a field, under a no-pivot readiness hypothesis. PLU proves a left-permuted factorization $P A=L U$ over a division ring, without that public condition. Schur triangularization \cite{schur1909charakteristischen,horn2012matrix} has two public surfaces: \texttt{HasSchur} states algebraic triangularization through an invertible similarity, while \texttt{HasUnitarySchur} states the complex unitary version with the inverse written as a conjugate transpose. Normal spectral results use implication-shaped predicates whose conclusion is a unitary diagonalization only under a normality hypothesis. Thus the same square recursive environment supports several genuinely different kinds of theorem before one even asks how the common framework assembles them.

The same square route also supports statements that are not elimination theorems in the narrow sense. For a positive-definite matrix over the reals (formally, an \texttt{RCLike} field with trivial star), the LDL development proves a strengthened LDL target with a unit lower-triangular factor and a positive diagonal factor \cite{golub2013matrix}. The Cholesky theorem is derived from this strengthened LDL theorem and states, under the same positive-definiteness hypothesis, the standard positive-diagonal Cholesky target $A=C C^{T}$ with $C$ lower triangular and positive on the diagonal. This derivation reuses the recursive work already performed for LDL rather than introducing a second Cholesky-specific induction.

This LDL/Cholesky development is a useful point of comparison with \mathlib, which already proves an LDL factorization \href{\mathliburl/Mathlib/LinearAlgebra/Matrix/LDL.lean\#L111}{\texttt{LDL.lower\_conj\_diag}} via Gram--Schmidt orthogonalization: for a positive-definite $S$ over an \texttt{RCLike} field it gives $S=L D L^{H}$ with $L$ lower triangular and $D$ diagonal. The present LDL theorem proves the same equational shape through the recursive head-tail driver rather than Gram--Schmidt, and additionally derives the canonical positive-diagonal Cholesky factorization $A=C C^{T}$, which \mathlib does not state. The two developments are complementary rather than competing: the \mathlib proof is more general in scalar scope, covering complex Hermitian matrices, whereas the present one is currently restricted to the real, trivial-star setting but yields the stronger Cholesky surface and arises as one instance of the same driver that produces the other 16 families. Relaxing the trivial-star restriction to recover the complex case is a natural extension. These examples show that the theorem surface may be a direct recursive factorization, a bridge from a stronger intermediate theorem, an implication-shaped spectral statement, or a similarity statement, while the driver still sees a universe-level predicate and proof hooks.

\subsection{Orthogonal and Unitary Reductions}

The orthogonal and unitary branch keeps the same recursive organization but changes the transport equation. QR is a one-sided factorization theorem: the local step clears the first column, the slice is a trailing square block, and the recursive result is lifted to a statement $A = QR$. Hessenberg and tridiagonalization are different. Their local step is expressed through similarity or unitary similarity, and the lifted object is a structured matrix rather than a pair of factors reconstructing the input directly. The distinction affects both the schema equation and the proof that the target survives transport along the chosen orthogonal or unitary transformation.

\begin{lstlisting}[caption={QR schema and orthogonal Hessenberg (\href{\repourl/MatDecompFormal/Instances/QR/Details.lean}{\texttt{Instances/QR/Details.lean}}; \href{\repourl/MatDecompFormal/Instances/OrthogonalHessenberg/Real.lean}{\texttt{Instances/OrthogonalHessenberg/Real.lean}})}]
def QR_Schema [LinearOrder ι] : DecompositionSchema ι ι R where
  Factors := Matrix ι ι R × Matrix ι ι R
  property := fun (Q, R') => IsOrthogonalMatrix Q ∧ IsUpperTriangular R'
  equation := fun A (Q, R') => A = Q * R'

def HasOrthogonalHessenberg (A : Matrix ι ι ℝ) : Prop :=
  ∃ Q : Matrix ι ι ℝ, ∃ H : Matrix ι ι ℝ, IsOrthogonalMatrix Q ∧ IsUpperHessenberg H ∧ A = Q * H * Qᵀ
\end{lstlisting}

The contrast in the listing is the main mathematical point of this family. QR proves a one-sided factorization theorem, whereas Orthogonal Hessenberg proves a similarity theorem. Tridiagonalization uses the same broad orthogonal or unitary setting, but its public target is implication-shaped because Hermitian structure is treated as part of the predicate. The difference among these theorems is therefore visible in the final proposition, not merely in whether a local step uses a Householder reflection or a similarity transformation. The QR lift and transport theorems give concrete examples of the corresponding proof obligations. Given a matrix $A$ that is QR-ready, meaning that its first column is already in Householder normal form, and a recursive QR witness for the tail submatrix, the lift produces a QR witness for $A$. The subsequent transport theorem absorbs a Householder left factor back into the orthogonal component.
\begin{lstlisting}[caption={QR recursive lift (\href{\repourl/MatDecompFormal/Instances/QR/Recursive.lean}{\texttt{Instances/QR/Recursive.lean}})}]
theorem qrReady_headTailSubmatrixLift (A : Matrix ι ι R) (hA : QRReady ι A) (hP : HasQR (A.submatrix (fun i : QRTailIdx ι => headTailEquiv.symm (Sum.inr i)) (fun j : QRTailIdx ι => headTailEquiv.symm (Sum.inr j)))) : HasQR A
\end{lstlisting}
\begin{lstlisting}[caption={QR transport (\href{\repourl/MatDecompFormal/Instances/QR/Recursive.lean}{\texttt{Instances/QR/Recursive.lean}})}]
theorem qr_transport_of_orthogonal_left (H A : Matrix ι ι R) (hH : IsOrthogonalMatrix H) (hQR : HasQR (H * A)) : HasQR A
\end{lstlisting}

Householder and Givens modules \cite{householder1958unitary,givens1958computation} refine the QR branch by adding product predicates such as \texttt{HasHouseholderQR} and \texttt{HasGivensQR}. These are stronger than the plain theorem \texttt{HasQR}, but they should not automatically be described as full numerical traces. A prior formalization of QR decomposition via Gram-Schmidt orthogonalization exists in Isabelle/HOL \cite{divason2015qr}; the present development uses Householder and Givens constructions instead. In the current code, there are explicit forgetful theorems from the stronger structured predicates back to \texttt{HasQR}, and this is the right way to describe the relation: the product-style variants retain more information about the orthogonal factor, while the basic QR theorem remains an existence theorem for an orthogonal factor and an upper-triangular factor.

Boundary-style reductions appear most clearly in Hessenberg \cite{wilkinson1965algebraic}, orthogonal/unitary Hessenberg, and tridiagonalization. Here the recursive invariant includes an active boundary column, so the local state is richer than an ordinary trailing submatrix. Tridiagonalization adds one more boundary condition: the square-driver target is implication-shaped, because the public theorem is meant for Hermitian matrices. This is a good example of statement design controlling reuse. The framework still proves one universe-level theorem, but the property encodes the precondition and determines which boundary data must be preserved by lifting. The shared driver handles descent; the instance remains responsible for the stronger invariant.

\subsection{Rectangular Reductions}

The rectangular branch is where the general indexed interface becomes essential rather than decorative. A recursive step may remove a row index, a column index, or both, and the lifted factors need not live over matching ambient types. This is why the framework uses \texttt{RectUniverse} and \texttt{RectStrategyData} instead of treating rectangular results as a square special case with padding. Gauss rank normal form is the cleanest example of this path. The theorem states a two-sided invertible equivalence to a data-oriented rank-normal-form matrix, and the recursive lift is correspondingly two-sided. Gauss is therefore a good representative of what the rectangular driver does without spectral or metric overhead.

\begin{lstlisting}[caption={Gauss rank normal form and SVD (\href{\repourl/MatDecompFormal/Instances/Gauss/Details.lean}{\texttt{Instances/Gauss/Details.lean}}; \href{\repourl/MatDecompFormal/Instances/SVD/Existence.lean}{\texttt{Instances/SVD/Existence.lean}})}]
def HasGaussRankNormalForm (A : Matrix m n R) : Prop :=
  ∃ P : Matrix m m R, ∃ Q : Matrix n n R, ∃ G : Matrix m n R, GaussInvertibleMatrix P ∧ GaussInvertibleMatrix Q ∧ IsGaussRankNormalForm G ∧ G = P * A * Q

theorem exists_svd (A : Matrix m n ℂ) : HasSVD A
\end{lstlisting}

The first declaration shows the cleanest rectangular theorem surface: two-sided invertible equivalence to a structured rectangular matrix. SVD strengthens that target by asking for unitary factors and a nonnegative diagonal middle matrix \cite{golub1965calculating}. Bidiagonalization changes the structured target again, replacing diagonal data by bidiagonal data. UTV keeps the same broad two-sided unitary form but uses a genuine rectangular upper-triangular, or upper-trapezoidal, middle factor. Its current one-step oracle is obtained from SVD head-basis data, so it is best described as a framework-routed UTV existence theorem rather than as a constructive UTV reduction algorithm. What unifies this family is not a single formula, but a common rectangular setting in which row and column structure are allowed to vary independently.

This distinction matters for the paper's claims. The rectangular branch shows that the driver handles independent row and column descent, but it does not imply that every rectangular theorem records the same kind of algorithmic data. Gauss is the clean driver instance. SVD and related unitary results are stronger existence theorems whose proofs may use spectral or noncomputable choices while still fitting the same rectangular interface. Their commonality lies in the shape of the recursive contract rather than in the computational source of each local witness. As a concrete example of a rectangular transport hook, the SVD transport absorbs two-sided unitary factors:
\begin{lstlisting}[caption={SVD transport (\href{\repourl/MatDecompFormal/Instances/SVD/Details.lean}{\texttt{Instances/SVD/Details.lean}})}]
theorem svd_transport_twoSidedUnitary (U : Matrix m m ℂ) (V : Matrix n n ℂ) (A B : Matrix m n ℂ) (hU : IsUnitaryMatrix U) (hV : IsUnitaryMatrix V) (hB : B = Uᴴ * A * V) (hSVD : HasSVD B) : HasSVD A
\end{lstlisting}

\subsection{Algebraic and Canonical Bridges}

The bridge branch uses the driver less as a step-by-step elimination engine and more as an interface boundary. The matrix theorem is still stated in the same schema style, but a large part of the mathematical work is delegated to PID, spectral, or canonical-form infrastructure. This is visible already in the public predicates: some record two-sided equivalence to ordered normal-form data, while others state similarity to a canonical representative whose existence is proved by a deeper bridge. In these cases the reusable matrix-facing interface organizes the final theorem, while the bridge supplies existence data that should not be described as a local executable reduction.

\begin{lstlisting}[caption={Smith normal form data and rational canonical form (\href{\repourl/MatDecompFormal/Instances/Smith/Details.lean}{\texttt{Instances/Smith/Details.lean}}; \href{\repourl/MatDecompFormal/Instances/RationalCanonical/Details.lean}{\texttt{Instances/RationalCanonical/Details.lean}})}]
structure SmithNormalFormData (D : Matrix m n R) where
  r : Type*
  order : Fin (Fintype.card r) ≃ r
  -- (additional fields: row, col, diag, injectivity,
  --  entry constraints omitted for brevity)
  divides_chain : ∀ k : Fin (Fintype.card r), (hnext : (k : Nat) + 1 < Fintype.card r) → diag (order k) ∣ diag (order ⟨(k : Nat) + 1, hnext⟩)

def HasRationalCanonical (A : Matrix ι ι K) : Prop :=
  ∃ P : Matrix ι ι K, ∃ Pinv : Matrix ι ι K, ∃ C : Matrix ι ι K, HasMatrixInverse P Pinv ∧ IsRationalCanonicalMatrix C ∧ A = P * C * Pinv
\end{lstlisting}

The first declaration shows that the Smith target carries ordered invariant-factor data through an explicit ordering and adjacent divisibility chain. The second declaration shows the shape of a canonical-form theorem in this development: a change of basis, an explicit inverse witness, and a canonical matrix satisfying its own structural predicate. These are not elimination theorems in the narrow sense. They are matrix-facing statements that sit at the boundary between the common decomposition language and imported algebraic structure. Consequently, their proof route is classified as bridge-derived even though their final statements use the same schema vocabulary as the explicit recursive instances.

Smith normal form \cite{smith1861systems} sits near the rectangular side of the library because it is built from two-sided row and column equivalence. Its current statement boundary is still delicate: \breaktt{IsSmithNormalForm} is data-oriented and records an ordered divisibility chain, but the theorem does not claim uniqueness, normalization of associates, or a constructive Smith elimination trajectory. Rational canonical form and Jordan-type theorems \cite{horn2012matrix,thiemann2016jordan} are square similarity statements, but most of their proof burden lies in block strategies and canonical-form bridges rather than in local elimination. The resulting proof object should therefore be described as a structured existence proof, not as an executable canonical-form algorithm. These bridge instances are still important evidence for the framework, but in a different way from PLU or QR. Their main interest in the present section is the kind of theorem they state: matrix-level existence results whose public predicates remain sensitive to how much bridge structure is imported. The next section turns from these family differences to the common assembly route itself.

\section{From Framework to Decomposition Theorems}
\label{sec:plu-case}

Section~\ref{sec:driver-patterns} compared theorem families by their public statements. We now explain how the common framework produces such theorems. The key point is that the reusable part is not a single decomposition formula, but an assembly path from a public target predicate to strategy data, proof hooks, a generic induction instance, and finally a theorem for an arbitrary matrix. We use PLU as the main worked example because it shows the full square assembly path most transparently, and then indicate how the same route changes in QR, Gauss, and bridge-heavy square instances without conflating their distinct proof routes and targets.

PLU remains the most informative worked example because it exercises almost every moving part of the square driver while still having a familiar mathematical statement. For a square matrix over a division ring, the theorem produces matrices $P$, $L$, and $U$ such that $P A=L U$, where $P$ is a permutation matrix, $L$ is unit lower triangular, and $U$ is upper triangular. Unlike LU or plain QR, PLU also forces a genuine branch in the reachability step: the recursive path depends on whether the first column is already zero or whether a pivot must first be moved to the head position. This makes PLU the clearest place to see how the abstract framework absorbs a branching recursive proof rather than a single fixed reduction. Its assembly begins with a public theorem surface and a universe-level target predicate. The type-indexed schema chooses a triple of square matrices as factors, records the permutation and triangularity predicates as the property field, and records $P A=L U$ as the equation field. The universe-level predicate then lifts that matrix theorem to the square universe on which the generic induction machinery operates.

\begin{lstlisting}[caption={PLU schema and universe predicate (\href{\repourl/MatDecompFormal/Instances/PLU/Details.lean}{\texttt{Instances/PLU/Details.lean}})}]
def PLU_Schema : DecompositionSchema ι ι R where
  Factors := Matrix ι ι R × Matrix ι ι R × Matrix ι ι R
  property := fun (P, L, U) => IsPermutation P ∧ IsUnitLowerTriangular L ∧ IsUpperTriangular U
  equation := fun A (P, L, U) => P * A = L * U

def PLU_P (x : SquareUniverse R) : Prop := HasPLU x.A
\end{lstlisting}

This distinction matters because the recursive slice lives over a tail index type, while the final theorem must return to the original matrix type after lifting and transport. The framework therefore does not recurse directly on a theorem of the form \lean{forall A, HasPLU A}. It recurses on a universe-level predicate, and the matrix theorem is recovered only at the end of the assembly. The next layer is the strategy layer, where the framework separates generic structure from instance-specific proof data. The head-tail strategy core fixes how the matrix is prepared, what counts as a sliceable state, and how the recursive slice is extracted. The PLU-specific record does not repeat that machinery; it supplies only the theorem-facing hooks for transport along the local relation and lifting a recursive slice witness back to the prepared matrix. These hooks are packaged once as strategy data.
\begin{lstlisting}[caption={PLU strategy data (\href{\repourl/MatDecompFormal/Instances/PLU/Driver.lean}{\texttt{Instances/PLU/Driver.lean}})}]
noncomputable def plu_strategy_proof :
    SquareStrategyProofData R PLU_P plu_strategy_core where
  transport := ...
  lift := ...

noncomputable def plu_strategy_data : SquareStrategyData R PLU_P :=
  mkSquareStrategyData plu_strategy_core plu_strategy_proof
\end{lstlisting}
The first field records the two PLU-specific obligations that the driver cannot synthesize by itself. The transport hook explains how a PLU witness is pulled back along the local relation used by the strategy core. The lift hook explains how a PLU witness for the recursive slice is enlarged to a PLU witness for the prepared matrix. Reachability and descent are then derived from the strategy core, while theorem reconstruction is handled by the generic square assembler. This division keeps the branching pivot algebra in the instance and prevents the generic driver from depending on a particular factorization equation or readiness predicate.

From the framework's viewpoint, the heart of the PLU example is the branching reachability obligation. The generic square interface does not ask for a proof of PLU in one step. It asks for the following data: given a non-base positive square universe, exhibit a related state that is sliceable and whose slice has smaller measure. In PLU this obligation is discharged by the textbook dichotomy. If the first column already vanishes below the head position, the matrix is ready for a tail-block reduction. Otherwise a pivot is available somewhere in that column; a row swap moves it to the head position, after which elimination produces a Schur-type recursive state. The two branches therefore differ in local algebra, but they satisfy the same abstract contract with the framework: reach a prepared matrix, prove strict descent, and expose a recursive slice.

The zero-column branch is the simpler of the two, and it shows clearly what the lift hook is for. When the first column is already zero, the recursive slice is just the trailing block. The reachability relation may then be realized by the identity step, so transport adds no new algebraic work. The substantive obligation is the lift theorem: starting from a PLU decomposition of the tail block, one must embed the smaller permutation, lower-triangular, and upper-triangular factors into the lower-right corner of larger block matrices and verify the full equation $P A=L U$. This is also where the reindexing layer from Section~\ref{sec:core-tech} becomes visible. The recursive witness lives over the tail index type, but the lifted factors must satisfy triangular predicates on the original square index type, so the proof passes through explicit head-tail equivalences and their order-preserving variants.

The pivot branch is structurally richer, and it is the part of PLU that most clearly exhibits the separation between lift and transport. Suppose the first column is not already zero. Reachability chooses an index $t$ of a nonzero pivot and forms the swapped matrix $B=(\mathrm{swap}_{0t})A$. At the level of the driver relation, this is the only transport information that matters: $B$ is locally related to $A$, and the recursive slice extracted from $B$ will be smaller than $A$. The PLU-specific algebra then proves more. After the swap, elimination clears the entries below the pivot, so the prepared matrix has a block form whose trailing block is the Schur complement. The recursive premise is therefore a PLU theorem for that Schur complement, not for the swapped matrix itself. The lift theorem reconstructs a PLU witness for the prepared matrix from the recursive witness on the Schur complement. Only after that does transport act: the helper theorem \texttt{hasPLU\_of\_left\_swap} absorbs the initial row swap into the permutation factor, converting a PLU decomposition of $(\mathrm{swap}_{0t})A$ into one of $A$.

This division of labor is important. The lift theorem is responsible for the block algebra that combines elimination data with the recursively obtained factors. Transport is responsible only for moving the theorem across the local relation used by reachability. In the zero-column branch the local relation can be equality, so transport is trivial. In the pivot branch the local relation is a single left row swap, so transport changes the permutation factor but leaves the rest of the reconstructed witness intact. The same abstract driver can accommodate both behaviors because it only requires a local relation and a proof that the target predicate is stable under it.

The base case is equally explicit. On a subsingleton or zero-dimensional index type, the witness is the trivial triple $(1,1,A)$. Formally this appears in two places. The public theorem uses \texttt{base\_plu\_subsingleton} to dispatch the degenerate case for an arbitrary square matrix. The generic assembler consumes a universe-level theorem \texttt{plu\_base\_univ}, which phrases the same idea on \texttt{SquareUniverse}. This matters because the generic recursion is not an induction on a natural parameter alone. The base test is performed on packaged matrices whose index types may vary, so the framework must recognize base objects through the universe measure rather than through a literal occurrence of \texttt{Fin 0}. The top-level theorem is short precisely because all local obligations have already been packaged as strategy data.

\begin{lstlisting}[caption={PLU existence theorem (\href{\repourl/MatDecompFormal/Instances/PLU/Existence.lean}{\texttt{Instances/PLU/Existence.lean}})}]
theorem exists_plu_decomposition (A : Matrix ι ι R) : HasPLU A := by
  by_cases h_sub : Subsingleton ι
  · exact base_plu_subsingleton A
  · let stepData : SquareStrategyData R PLU_P := plu_strategy_data
    let inst : SquareSubtypeInductionInstance R :=
      mkSquareSubtypeInductionInstanceFromStrategy PLU_P plu_base_univ stepData
    have hP : inst.P (SquareUniverse.ofMatrix A) :=
      SquareSubtypeInductionInstance.prove_for_matrix (inst := inst) A
    exact hP
\end{lstlisting}

This listing is the code-level expression of the assembly path. The theorem first isolates the subsingleton case because that case is mathematically trivial and does not require recursion. In the nontrivial case, \texttt{plu\_strategy\_data} packages the branching reachability theorem, the strict-descent proof, and the PLU-specific lift and transport hooks. The constructor \href{\repourl/MatDecompFormal/Framework/DecompositionDriver.lean}{\breaktt{mkSquareSubtypeInductionInstanceFromStrategy}} turns those local ingredients into a value of the generic subtype-induction interface. The theorem \href{\repourl/MatDecompFormal/Framework/UniverseDecomposition.lean}{\breaktt{SquareSubtypeInductionInstance.prove_for_matrix}} then applies the reusable recursive argument to the concrete matrix. The final theorem no longer mentions row swaps, elimination matrices, Schur complements, or head-tail reindexings because those details have already been absorbed into the strategy data and the generic constructor.

At this point the reusable pipeline is visible, and PLU is only one instance of it. In QR, the universe-level target changes from \texttt{HasPLU} to \texttt{HasQR}, and the lift and transport hooks prove preservation of orthogonality and upper triangularity rather than permutation and unit lower triangularity. In LDL and Cholesky, the square route first produces strengthened LDL data and then converts it to a positive-diagonal Cholesky factor. In Gauss rank normal form and UTV, the same route passes through \texttt{RectUniverse} and \texttt{RectStrategyData}, and the public theorem becomes a two-sided statement over independent row and column index types. In unitary Schur, rational canonical, and Jordan-type developments, the square shell persists, but the strategy hooks interface with spectral, block, or canonical-form bridges. What is reused across all these cases is the passage from theorem surface, target predicate, strategy data, and generic assembler to a public theorem. What changes is the local mathematics required to supply those ingredients.

PLU is therefore representative for the framework in a precise but limited sense. It is the clearest square example in which reachability genuinely branches, lifting requires explicit block algebra, and transport changes the final witness in a nontrivial but transparent way. It does not stand for the whole instance library, but it does exhibit the full square assembly path from schema and universe-level predicate to strategy data, generic construction, and final theorem. PLU therefore remains the best worked example for explaining how the framework produces decomposition theorems while leaving the spectral and bridge-derived routes to the broader comparative discussion that follows.

\section{Statement Boundaries}
\label{sec:statement-boundaries}

The instance library should not be read as a list of uniformly computational algorithms or uniformly strong theorem statements. The common driver packages a shared recursive proof shape, not a promise that every public predicate records the same amount of step data or matches the strongest classical theorem associated with its file name. Some instances are close to explicit elimination proofs, some record only an existence route compatible with the same driver interface, and some use bridge theorems whose mathematical content is stronger than any local elimination step. Table~\ref{tab:instance-landscape} summarizes the main groups together with the most important statement boundary in each group.

\begin{table}[t]
\centering
\caption{Representative, non-exclusive instance groups and their principal statement boundaries.}
\label{tab:instance-landscape}
\begin{tabular}{@{}
  >{\raggedright\arraybackslash}p{0.23\textwidth}
  >{\raggedright\arraybackslash}p{0.24\textwidth}
  >{\raggedright\arraybackslash}p{0.33\textwidth}
  @{}}
\toprule
Family / proof pattern & Representative examples & Principal statement boundary \\
\midrule
Square elimination and triangularization & PLU, LU, LDL, Cholesky & Readiness hypotheses and bridge-derived targets differ across the group \\
\addlinespace[2pt]
Orthogonal, unitary, and boundary reductions & QR variants, Hessenberg variants, Schur, normal spectral decomposition, tridiagonalization & Plain existence, implication-shaped targets, and structured variants retain different data \\
\addlinespace[2pt]
Rectangular two-sided reductions & Gauss rank normal form, SVD, bidiagonalization, UTV & Independent row/column descent; spectral and bridge routes are not executable traces \\
\addlinespace[2pt]
Canonical and algebraic bridges & Smith normal form, rational canonical form, Jordan-type forms & Existence metadata without a general claim of uniqueness, normalization, or executable reduction \\
\bottomrule
\end{tabular}
\end{table}

The first distinction is between existence, proof path, and execution. A theorem such as PLU or the basic QR theorem records an existence statement together with a proof route that is compatible with explicit local transformations. A theorem phrased through Householder or Givens product predicates may retain more structure about the orthogonal factor, but it still should not be confused with a fully executed numerical routine unless the step data is part of the final witness. At the other extreme, a theorem proved through a spectral theorem, PID bridge, canonical-form bridge, or classical choice may still fit the same driver architecture, but the resulting proof object records an existence route rather than a concrete numerical trace. The common interface is therefore stable across these examples, but the meaning of the final witness varies with the target predicate and with the local mathematics used to supply it.

The second distinction is between theorem names and actual predicates. Lean proves exactly the proposition that has been stated, not the name that a file happens to suggest. This is the point at which statement design becomes a mathematical issue rather than a presentation detail, because a weak or data-oriented predicate may still support a correct theorem while falling short of the classical theorem one might first expect from the surrounding terminology. The following three definitions illustrate different boundary phenomena: UTV uses a rectangular upper-triangular middle factor obtained through SVD-derived head-basis data; Schur triangularization has a separate complex unitary target rather than overloading algebraic invertible similarity; and the tridiagonalization driver places its Hermitian hypothesis inside an implication-shaped property.

\begin{lstlisting}[caption={Statement boundary examples (\href{\repourl/MatDecompFormal/Instances/UTV/Details.lean}{\texttt{Instances/UTV/Details.lean}}; \href{\repourl/MatDecompFormal/Instances/Schur/Details.lean}{\texttt{Instances/Schur/Details.lean}}; \href{\repourl/MatDecompFormal/Instances/Tridiagonalization/Details.lean}{\texttt{Instances/Tridiagonalization/Details.lean}})}]
def IsRectangularUpperTriangular (T : Matrix m n R) : Prop :=
  ∀ i j, colRank n j < rowRank m i → T i j = 0

def HasUnitarySchur (A : Matrix ι ι ℂ) : Prop :=
  ∃ Q : Matrix ι ι ℂ, ∃ T : Matrix ι ι ℂ, IsUnitaryMatrix Q ∧ IsUpperTriangular T ∧ A = Q * T * Qᴴ

def Tridiagonalization_P (x : SquareUniverse ℂ) : Prop :=
  x.A.IsHermitian → HasUnitaryTridiagonalization x.A
\end{lstlisting}

Returning to the second distinction, several boundaries concern names whose mathematical meaning depends on the exact predicate. Cholesky is now stated as a positive-diagonal lower-triangular Cholesky target and is derived from strengthened LDL data, so it should be described as a Cholesky theorem obtained through an LDL bridge, not as an LDL-only statement. UTV now has a rectangular upper-triangular middle factor, so the remaining caveat is not the target predicate but the proof route: the current existence theorem is assembled from an SVD-derived one-step oracle rather than from a constructive UTV reduction trace. Smith normal form introduces another kind of boundary. The current predicate records ordered invariant-factor data and adjacent divisibility, but the theorem remains a data-oriented existence result; it does not claim uniqueness of the invariant factors up to associates, a chosen normalization convention, or an executable Smith reduction algorithm.

Returning to the first distinction, a further group of boundaries concerns how much local transformation data survives into the final witness. Product-style predicates such as \texttt{HasHouseholderQR} and \texttt{HasGivensQR} retain more structure than the plain theorem \texttt{HasQR}, but they should still be described through the predicate actually proved rather than through an informal promise that every elementary reflector or rotation has been stored as an executed numerical trace. The same caution applies to bidiagonalization and tridiagonalization. Their public theorem surfaces are valid and useful, but a proof that passes through spectral or bridge-style infrastructure should not automatically be read as a formalization of a specific Golub--Kahan or Householder run. Rational canonical and Jordan theorems are stronger examples of this phenomenon: they prove substantial canonical-form existence results, yet much of the mathematical burden lies in imported block and canonical-form bridges rather than in a local elimination routine.

The same care is needed elsewhere. \texttt{HasSchur} is an invertible-similarity triangularization theorem, while \texttt{HasUnitarySchur} is the separate complex unitary theorem. Both \breaktt{Tridiagonalization_P} and \texttt{NormalSpectral\_P} are implication-shaped: the Hermitian or normality hypothesis is carried inside the target predicate, so the theorem is vacuously true when the hypothesis does not hold. Product-style QR, bidiagonalization, and Hessenberg variants expose product or boundary data, but those predicates still need to be read as the data they state, not as claims that a complete numerical execution log has been extracted. Which predicate is proved determines what the final theorem really says, what kind of witness later automation may reuse, and how one should compare the formal result with its textbook namesake.

Seen from this perspective, statement boundary is not a minor editorial issue appended after the proofs are done. It is one of the central research outcomes of the development. The reusable interface determines which recursive arguments can be shared, but the public predicate determines what theorem is actually obtained from that interface. If the predicate is intentionally weak, Lean faithfully proves a weak theorem. If the predicate is strengthened, the lift, transport, and bridge obligations become correspondingly stronger. This is why theorem design and proof reuse have to be discussed together: the framework tells us how to prove a large family of decomposition theorems, but the predicates decide how close those theorems come to the classical mathematical statements one wants to formalize.

\section{Quantitative Overview}
\label{sec:quantitative}

Table~\ref{tab:quantitative} summarizes the development in terms of physical source lines, proof route, target form, and framework reuse. All counts use the fixed 17-family audit scope reported here. LOC counts physical Lean source lines, including comments and blank lines; direct shared imports count distinct Framework, Components, and Abstractions modules imported by each family. The proof-route axis identifies whether the principal theorem is assembled through explicit recursive steps, spectral existence results, or another decomposition or algebraic bridge. The target-form axis separately records the hypothesis shape of the public theorem and whether structured variants expose additional witness data. The counted scope comprises 37{,}337 physical lines of Lean~4 code across 118 files: 34{,}503 lines in the decomposition instances and 2{,}834 lines in the shared framework, abstractions, and components. A lexical source scan, instead of a Lean declaration query, finds 1{,}162 proofs (693 theorems and 469 lemmas), 725 definitions (516 of which are noncomputable), 128 structures, and 84 abbreviations under this scope across all counted files.

\begin{table}[!htbp]
\centering
\caption{Instance landscape: physical lines of code, proof route, target form, and framework reuse.}
\label{tab:quantitative}
\footnotesize
\setlength{\tabcolsep}{3pt}
\begin{tabular}{@{}lr>{\raggedright\arraybackslash}p{0.15\textwidth}>{\raggedright\arraybackslash}p{0.17\textwidth}c@{}}
\toprule
Instance & LOC & Proof route & Target form & \shortstack{Direct shared\\imports} \\
\midrule
PLU & 982 & Explicit recursive & Unconditional & 11 \\
LU & 1{,}469 & Explicit recursive & Conditional & 8 \\
LDL & 691 & Explicit recursive & Conditional & 7 \\
Cholesky & 168 & Bridge-derived & Conditional & 3\tnote{a} \\
\addlinespace[2pt]
QR (+ Householder, Givens) & 2{,}946 & Explicit recursive & Unconditional; structured variants & 10 \\
Hessenberg & 1{,}930 & Explicit recursive & Unconditional & 4 \\
Orth./Unit.\ Hessenberg & 2{,}980 & Explicit recursive & Unconditional; structured variants & 0\tnote{b} \\
Schur (algebraic + unitary) & 1{,}442 & Spectral existence & Unconditional & 7 \\
Normal Spectral & 1{,}683 & Spectral existence & Implication-shaped & 3 \\
Tridiagonalization & 2{,}160 & Spectral / explicit variants & Implication-shaped; structured variants & 3 \\
\addlinespace[2pt]
Gauss rank normal form & 1{,}142 & Explicit recursive & Unconditional & 3 \\
SVD & 1{,}688 & Spectral existence & Unconditional & 3 \\
Bidiagonalization & 2{,}344 & Spectral existence & Unconditional; structured variants & 4 \\
UTV & 831 & Bridge-derived & Unconditional & 0\tnote{b} \\
\addlinespace[2pt]
Smith normal form & 2{,}728 & Bridge-derived & Unconditional & 3 \\
Rational canonical form & 3{,}681 & Bridge-derived & Unconditional & 5 \\
Jordan-type forms & 5{,}638 & Bridge-derived & Conditional / unconditional variants & 4 \\
\midrule
Shared infrastructure & 2{,}834 & --- & --- & --- \\
\textbf{Total} & \textbf{37{,}337} & & & \\
\bottomrule
\end{tabular}
\begin{tablenotes}
\footnotesize
\item[a] Cholesky reuses LDL induction and has no direct Framework dependency.
\item[b] These families call the driver but obtain framework types through sibling-instance imports.
\end{tablenotes}
\end{table}

The two classification columns separate facts that the earlier single-axis terminology conflated. \emph{Explicit recursive} means that concrete local steps, slices, lifting, and transport enter the driver; it does not assert constructive logic or executable computation. \emph{Spectral existence} marks a route whose local existence data comes primarily from spectral results or noncomputable spectral choices. \emph{Bridge-derived} marks a theorem obtained from another decomposition, as with Cholesky and UTV, or from algebraic and canonical-form bridges, as with Smith, rational canonical, and Jordan forms. The target-form column separately records whether the principal theorem is unconditional, conditional on an explicit hypothesis, implication-shaped, or accompanied by variants whose final witnesses retain structured trace data. In particular, bidiagonalization is classified by its spectral existence route even though separate product and trace predicates expose additional structure. The shared infrastructure of 2{,}834 lines (about 7.6\% of the counted total) has a much larger architectural footprint than its size suggests, because it is reused across the family library rather than consumed once.

At the engine level, the two most central modules---\texttt{DecompositionDriver} (699~lines) and \texttt{HeadTail} (193~lines)---are each directly imported by 14 of the 17 instance families. The matrix-level entry \texttt{prove\_for\_matrix} occurs in 15 families; OrthogonalHessenberg receives and instantiates the relevant framework types transitively, while Cholesky inherits recursion from LDL. The two engine modules therefore provide 892 lines of the recursive backbone used throughout the library. At the component level, small reusable modules achieve high fan-out: \texttt{SubmatrixMethod} (69~lines) is directly imported by 11 families, \texttt{Properties.Triangular} (144~lines) by~6, and \texttt{Properties.Reindex} (159~lines) by~5. The ``Direct shared imports'' column in Table~\ref{tab:quantitative} records how many shared modules each family directly imports; the most coupled families (PLU~11, QR~10, LU~8) draw on the majority of the shared layer.

Beyond callable APIs, the framework imposes a recurring proof organization: all 17 families contain a \texttt{Details} module, 16 contain \texttt{Strategy} and \texttt{Existence} modules, and 14 contain a \texttt{Direct} module. The dominant workflow is to define the schema, fill strategy data, assemble the driver, and export a theorem; bridge-only families may inherit part of this workflow from another instance. The final export step is correspondingly small: 13 of the 16 families with a dedicated \texttt{Existence} module keep that module below 200 physical lines, including 32 lines for PLU and 44 for LDL, even when the supporting decomposition-specific algebra runs to thousands of lines. Normal, Schur, and Jordan exceed that threshold because their existence modules retain more spectral or canonical-form bridge material. This organization is a concrete payoff of the shared recursive backbone: once an instance supplies its local algebra and lift/transport hooks, the cross-type recursion does not have to be re-derived.

Three instances warrant specific comment on their framework coupling. Cholesky (note~a in Table~\ref{tab:quantitative}) is a pure algebraic bridge: it invokes the LDL existence theorem---which internally uses the full framework induction---and then applies a one-step diagonal square-root transformation. It needs \texttt{DecompositionSchema} from the Abstractions layer and property lemmas from Components, but no direct Framework import; the recursive complexity is borne entirely by LDL. Orthogonal/Unitary Hessenberg and UTV are marked by note~b because they receive framework types transitively through sibling-instance imports (Hessenberg and SVD, respectively) rather than through direct shared-module imports. Their own code nevertheless constructs induction instances and fills the strategy, lift, and transport data---UTV calling \texttt{prove\_for\_matrix} directly, while Orthogonal/Unitary Hessenberg drives the shared boundary-column subtype descent---so the transitive import is a Lean module-system artifact, not a weaker conceptual coupling to the framework. Individual instances range from 168~lines (Cholesky) to 5{,}638~lines (Jordan), while the architecture used to organize their proof routes is supplied by the shared layer.

\FloatBarrier
\section{Discussion and Future Work}
\label{sec:discussion}

The results support a precise view of generalization and reuse. When several theorems follow the same sequence of local preparation, strict descent, recursive application, lifting, and transport, the repeated routine should itself be stated as a more general theorem rather than copied as a proof script. Here that theorem-level interface combines universe packaging, strategy data, and proof hooks, while each instance supplies its own algebra and public predicate. The design principle of packaging algebraic structure into reusable interfaces \cite{garillot2009packaging} is thereby extended from individual hierarchies to families of decomposition proofs. In square elimination, the common interface packages head-tail recursion and block lifting; in rectangular problems, it accommodates independent row and column descent; and in bridge-heavy canonical forms, it supplies a matrix-facing theorem surface for deeper algebraic arguments. The quantitative evidence in Section~\ref{sec:quantitative} shows that this shared organization is used across the library even though the source of the final witnesses varies substantially.

The same comparison also shows why statement design is more important than the apparent difficulty of a proof script. One driver can support unconditional factorizations, implication-shaped results, structured trace variants, and bridge-derived existence theorems, but these results do not justify identical algorithmic claims. Explicit local data may certify a route from the original matrix to a recursive slice and back, whereas spectral, PID, or canonical-form bridges generally provide structured existence without an executable numerical procedure. The target predicate must therefore specify the factors, preserved properties, recursive domain, transport relation, and retained step data before automation is applied. Once this interface is mathematically appropriate, many remaining obligations become local lift, transport, readiness, and block-identity lemmas. Reusable tactics and AI-assisted proof search can then help discharge those obligations, but they do not replace the human task of choosing the theorem boundary. Section~\ref{sec:statement-boundaries} records these boundaries because the formal statement, rather than the theorem name or the length of its proof, determines what has actually been established.

The framework also has clear computational and mathematical scope boundaries. It targets decomposition proofs organized as head-tail recursive reductions, where each step removes one or more index positions and the recursive call concerns a strictly smaller matrix. This pattern covers elimination-based proofs, orthogonal reductions, and canonical-form arguments connected through algebraic bridges. It does not cover convergence-based algorithms such as QR iteration or Jacobi iteration, whose correctness depends on asymptotic analysis rather than dimension reduction, and randomized decompositions are likewise outside the current scope. Moreover, the present constructions use \texttt{Classical.choice} and are noncomputable. This choice fits the abstract interfaces and theorem libraries available in \mathlib and keeps the mathematical statements independent of implementation details, but it also means that the resulting witnesses are not directly executable numerical code. Executable extraction would require an additional refinement layer, as in CoqEAL \cite{denes2012refinement}, together with computable local searches and suitable data representations.

Future work should first separate the layer that records existence proof paths from a refinement layer that records executable algorithms, allowing both forms of result to coexist without conflating proof organization with numerical implementation. A second direction is to strengthen canonical metadata where the present predicates deliberately stop at existence. Relevant examples include uniqueness or ordering conventions for SVD data, normalization up to associates for Smith invariant factors, and richer block metadata for rational canonical and Jordan forms. A third direction concerns proof engineering: improved automation for block identities, reindexing, order transport, and structural predicates would reduce the decomposition-specific cost of new instances. These extensions would test whether the current interfaces remain stable as theorem statements become stronger and more computational. The present development does not settle those questions, but it establishes a common recursive foundation against which stronger statements, executable refinements, and automated proof support can be compared.

\bibliographystyle{unsrt}
\bibliography{bibliography}

\end{document}